\documentclass{article}
\usepackage{amsmath,amsfonts,latexsym}
\usepackage[english]{babel}
\usepackage[pdftex]{graphicx}
\usepackage{hyperref}
\usepackage{multicol}
\usepackage{amsmath}

\usepackage{amsthm}
\usepackage{amssymb}
\usepackage{amsfonts}
\usepackage{multicol}
\usepackage{color}
\usepackage{authblk}

\usepackage[all]{xy}

\usepackage{tikz}
    \usetikzlibrary{cd}
    \usetikzlibrary{shapes.misc}
    
 \usepackage{authblk}

\newtheorem*{theorem*}{Theorem}
\newtheorem{theorem}{Theorem}
\newtheorem{lemma}[theorem]{Lemma}
\newtheorem{corollary}[theorem]{Corollary}
\newtheorem{proposition}[theorem]{Proposition}
\newtheorem{definition}{Definition}
\newtheorem{example}{Example}

\usepackage{hyperref}
\hypersetup{
colorlinks=true,       
    linkcolor=blue,          
    citecolor=blue,        
    filecolor=blue,      
    urlcolor=blue           
}

\begin{document}
 
\title{The minima of the geodesic length functions of uniform filling curves}


\author[$\dagger$]{Ernesto Girondo}
\author[$\ddag$]{Gabino Gonz\'alez-Diez}
\author[$\star$]{Rub\'en A. Hidalgo}

\affil[$\dagger$, $\ddag$]{Departamento de Matem\'aticas, Universidad Aut\'onoma de Madrid and Instituto de Ciencias Matem\'aticas ICMAT (CSIC-UAM-UCM-UC3M). 28049-Madrid, Spain. \ \texttt{ernesto.girondo@uam.es, gabino.gonzalez@uam.es}}
\affil[$\star$]{Departamento de Matem\'atica y Estad\'{\i}stica, Universidad de La Frontera. Campus  Andr\'es Bello, Avenida Francisco Salazar  01145 CASILLA 54-D
Temuco, Chile. \texttt{ruben.hidalgo@ufrontera.cl}}

\renewcommand{\thefootnote}{\fnsymbol{footnote}} 
\renewcommand{\thefootnote}{\arabic{footnote}}

\maketitle

\begin{abstract}
There is a natural link between (multi-)curves that fill up a closed oriented surface and dessins d'enfants. We use this approach to exhibit explicitly the minima of the geodesic length function of a kind of curves (uniform filling curves) which include those that admit a homotopy equivalent representative such that all self-intersection points as well as all faces of their complement have the same multiplicity. We show that these minima are attained at the Grothendieck-Belyi surfaces determined by a natural dessin d'enfant associated to these filling curves. In particular they are all Riemann surfaces defined over number fields. 
\end{abstract}

\section{Introduction and statement of results}  \label{subsec:intro}

Geodesic length functions have been a classical subject in Teichm\"uller theory. We shall start by setting our notation and recalling some fundamental facts of this theory.

\

Let $X$ be a closed oriented surface of genus $g\ge 2$. A marking on $X$ is a pair $(f,R)$, where $R$ is a closed Riemann surface and $f:X\to R$ is an orientation-preserving homeomorphism. Two markings $(f_1, R_1)$ and $(f_2, R_2) $ are considered equivalent if there is an isomorphism $\Phi: R_1 \to R_2$ such that $f_2^{-1}\circ \Phi \circ f_1$ is isotopic to the identity map in $X$. The set of equivalence classes of markings is the Teichm\"uller space $T(X)$. It is well-known that $T(X)=T_g$ is a contractible complex manifold of dimension $3g-3$. The modular (or mapping class) group $\mathrm{Mod}_g$ (or $\mathrm{MCG}_g$) is the group of isotopy classes of orientation preserving homeomorphisms of $X$. It acts on $T_g$ by the rule
$$
\begin{array}{rcl}
\mathrm{Mod}_ g  \times  T_g & \longrightarrow & T_g \\
(h, [f,R]) & \longmapsto & [f \circ h^{-1}, R]
\end{array}
$$
and the quotient is the moduli space $\mathcal{M}_g$ of (isomorphy classes of) closed Riemann surfaces or, equivalently, complete complex algebraic curves, of genus $g$. We will say that a closed Riemann surface \emph{is defined over a number field} if it is the Riemann surface of an algebraic curve defined over a number field. Such Riemann surfaces give rise to a countable but dense subset of $\mathcal{M}_g$.  These definitions readily extend to the case of Riemann surfaces of genus $g$ with $n$ distinguished points to produce the analogous mapping class group $\mathrm{Mod}_{g,n}$ and Teichm\"uller and Moduli spaces $T_{g,n}$ and  $\mathcal{M}_{g,n}$ of complex dimension $3g-3+n$ (see e.g. \cite{Bers_1981}).

\

Given a homotopically non-trivial closed curve $\gamma \subset X$ and a marking $(f,R)$ one may consider the unique closed geodesic $\gamma_{(f, R)}\subset R$ in the same homotopy class  as  
$f(\gamma)$. Equivalent markings produce identical geodesics, and therefore the rule that to each point $[f,R]\in T_g$ associates the length  $\ell_{\gamma}([f,R])$ of the geodesic $\gamma_{(f, R)}$ produces a well-defined function $\ell_{\gamma}:  T(X)  \longrightarrow  (0, \infty)$, called the \emph{geodesic length function}. If instead of a single curve $\gamma$ we have a collection of (freely) homotopically non-trivial and pairwise homotopically distinct curves $\Gamma=\{ \gamma_1, \ldots , \gamma_r \}$ the length function $\ell_{\Gamma}$ is defined to be the sum of the lengths of the corresponding geodesics. In case $r>1$ we sometimes refer to $\Gamma$ as a \emph{multicurve}.

A famous result, first proved by Kerckhoff, states that if $\Gamma$ is a filling multicurve of $X$ then 
 the function $\ell_{\Gamma}$ attains a minimum in Teichm\"uller space and that  this minimum is unique (see \cite{Kerckhoff_1983}, \cite{Wolpert_1987} and the related work \cite{Bestvina_et_al_2013}), thereby inducing a unique hyperbolic structure on $X$. We say that a multicurve $\Gamma$ in minimal position fills up $X$, or simply that $\Gamma$ is a filling multicurve, if 
 every non-trivial simple loop has non-empty intersection with  some component $\gamma_i$ of $\Gamma$. When $\Gamma=\{ \gamma \}$ has only one component we say that $\gamma$ is a \emph{filling curve}. Figure \ref{fig_g2a} shows examples of curves that fill up the surfaces of genus 2 and 3.

\begin{figure}[!htb]
\centering{
\begin{tabular}{lcr}
\includegraphics[width=0.38\textwidth]{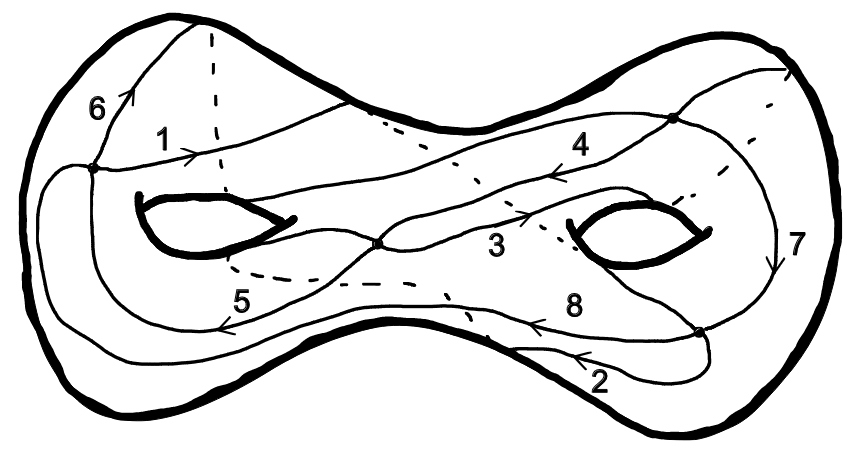} & \quad \quad 
& \includegraphics[width=0.5\textwidth]{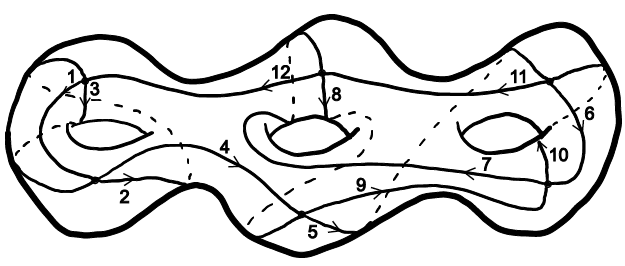}
\end{tabular}}
\caption{Curves that fill up the surfaces of genus 2 and 3. Consecutive numbers label consecutive arcs connecting self-intersection points of the curve.}\label{fig_g2a}
\end{figure}

A filling multicurve $\Gamma \subset X$ in minimal position divides the surface into a disjoint union of faces and so, by placing a (white) vertex between any two consecutive self-intersection points 
(which we regard as black vertices), $\Gamma$ is readily made into a bipartite graph of the kind Grothendieck called \emph{dessins d'enfants}. From a combinatorial point of view this dessin, we shall  denote it as $\mathcal{D}_{\Gamma}$, will be completely determined by two permutations $\sigma_0$ and $\sigma_1$ describing the cyclic ordering of the edges around the white and the black vertices respectively. Note that, by construction, the number $E=E(\Gamma)$ of edges of  $\mathcal{D}_{\Gamma}$ is twice the number $d$ of components of $\Gamma \smallsetminus \{\mbox{self-intersection points}\}$, which we  call the \emph{arcs} of $\Gamma$. It was observed by Grothendieck that such a dessin $\mathcal{D}_{\Gamma}$ endows the topological surface $X$ with a hyperbolic structure, usually referred to as the Grothendieck-Belyi surface of $\mathcal{D}_{\Gamma}$, we shall denote it as $S_{\Gamma}=\mathbb{D}/K$. Grothendieck-Belyi surfaces are defined over a number field and are uniformized by a  finite index  Fuchsian subgroup $K$ of the triangle group $\Delta(2,2m,k)$,  where $2m$ and $k$ are the least common multiples of the degrees of black vertices and faces of the graph $\mathcal{D}_{\Gamma}\subset X$. The group $K$ is determined by the combinatorics of $\Gamma$ as the inverse image of the one point stabilizer subgroup of the symmetric group $\mathbb{S}_{2d}$ via an obvious homomorphism $\Delta(2,2m,k) \longrightarrow \mathbb{S}_{2d}$ determined by the pair of permutations $\sigma_0, \sigma_1$.
In principle $K$ need not be torsion free but it will happen to be so in the case of the (uniform) filling multicurves we are primarily considering in this article.
(We recall that $\mathbb{D}/\Delta(2,2m,k)$ is the Riemann surface of genus 0 as its fundamental domain is a quadrilateral formed as the union of two copies of the triangle $T_{2,2m,k}$ with angles $\pi/2, \pi/2m$ and $\pi/k$).

Thus, from quite different points of view, the theories of dessins d'enfants and geodesic length functions allow us to associate to a filling multicurve  $\Gamma \subset X$ two Riemann surface structures on $X$. In this paper we investigate the relation between them. The results we obtain will enable us to explicitly determine the minima of the geodesic length function of a family of multicurves that include those which admit a homotopy equivalent representative such that their self-intersection points as well as their faces have the same multiplicity. To our knowledge there are no instances 
in the literature of geodesic length functions of filling curves whose minima have been explicitly identified.

More precisely, our main results are as follows:

\begin{enumerate}
\item Suppose that a filling multicurve of $X$ admits a homotopy equivalent representative $\Gamma=
\{ \gamma_1, \ldots , \gamma_r \}$ with $d$ arcs and that all its self-intersection points have a common degree $2m$ and all its faces carry  a common number $k$ of arcs. Then the minimum of the geodesic length function $\ell_{\Gamma}$ is realized by the Grothendieck-Belyi surface $S_{\Gamma}=\mathbb{D}/K$ and the geodesic $\beta^{-1}(L_{m,k})\subset S_{\Gamma}$, where $\beta: \mathbb{D}/K \longrightarrow \mathbb{D}/\Delta(2,2m,k)$ stands for the quotient map induced by the group inclusion $K < \Delta(2,2m,k)$ and $L_{m,k}$ is the projection to $\mathbb{D}/\Delta(2,2m,k)$ of the edge of the triangle $T_{2,2m,k}$ opposite to the angle $\pi/k$. The length of this geodesic being 
$$
\min_{[f,S]\in T_g} \ell_{\Gamma}([f,S])= 2d \cdot (\mbox{length of } L_{m,k})
$$ 
or, equivalently, $d$ times the length of the side of the regular $k$-gon of angle $\pi/m$ (see Theorem \ref{th:main}).
\item Conversely, any genus $g$ surface subgroup $K<\Delta(2,2m,k)$ uniformizes a Riemann surface $S=\mathbb{D}/K$ which realizes the minimum of the geodesic length function $\ell_{\Gamma}$ of some filling multicurve $\Gamma$. In fact, with the notation as in point 1, the multicurve $\Gamma$ can be taken to be $\beta^{-1}(L_{m,k})\subset S$ and $S$ will then agree with the Grothendieck-Belyi surface $S_{\Gamma}$.

Moreover, if $\sigma_0$, $\sigma_1$ are the two permutations representing the dessin $\mathcal{D}_{\Gamma}$, the number $r$ of components of $\Gamma$ coincides with half of the number of disjoint cycles of $\sigma_1^m \sigma_0$ (Theorem \ref{th:comb}).
\item Statements similar to 1 and 2 hold if one replaces the triangle group $\Delta(2,2m,k)$ by a triangle group of the form $\Delta(2l,2m,k)$ (see Proposition \ref{pr:destomulticurve_2}).
\item For each genus $g$, the groups $\Delta(2,4,8g-4)$ and $\Delta(2,4,4g)$  contain a surface subgroup $K$ of genus $g$ such that the multicurve referred to in 2
produces a filling curve (that is, with only one component) in general position whose corresponding length function attains its minimum at the hyperbolic surface $S=\mathbb{D}/K$. The complement of this curve has one component in the first case and two in the second one 
 (Theorem \ref{th:exist}).
\end{enumerate}

\section{A quick review of the Grothendieck-Belyi theory of dessins d'enfants} \label{sec:review_dessins}

In this section we summarize the basics of the theory of dessins d'enfants and Belyi surfaces. We refer to  \cite{Girondo_Gonzalez-Diez_2012,Jones_and_Wolfart_2016} for a more detailed exposition.

\

A \emph{dessin d'enfant} is a connected bipartite finite graph $\mathcal{D}$ embedded in a closed oriented surface $X$ in such a way that the connected components of $ X \smallsetminus \mathcal{D}$, which are called \emph{faces}, are 2-cells, i.e.  homeomorphic to open discs.  The vertices of a dessin can be coloured white or black in such a way that the two  vertices of any edge do not have the same colour. A vertex $v$ has degree $m$ if the number of edges incident to $v$ is precisely $m$, whereas a face (or 2-cell) $C$ has degree $k$ if its boundary $\partial C$ consists of $2k$ edges of $\mathcal{D}$ (where an edge $e$ must be counted twice if it meets the same face at both sides, equivalently if no other face contains $e$ in its boundary). Two dessins d'enfants are called {\it equivalent} if there is an orientation-preserving homeomorphism of $X$ that induces a colour-preserving isomorphism of the corresponding bipartite graphs.

\

The \emph{passport} of a dessin ${\mathcal D} \subset X$ is the tuple
$(a_{1},\ldots,a_{W};b_{1},\ldots,b_{B};c_{1},\ldots,c_{F})$ given by the degrees of the white vertices, black vertices and faces, respectively.  
 If $a, b$ and $c$ are the the least common multiples of $\{a_{i} \} $, $\{ b_{j} \}$ and $\{ c_{k} \}$, then we 
  say that ${\mathcal D}$ has \emph{type} $(a,b,c)$.    
 The \emph{degree} $\mathrm{deg}(\mathcal{D})$ of the dessin is the number $E$ of edges of $\mathcal{D}$, which is related to the degrees of vertices and faces by the relations   
$E=a_{1}+\cdots+a_{W}=b_{1}+\cdots+b_{B}=c_{1}+\cdots+c_{F}$. 
As a consequence of Euler's formula, the genus $g$ of $X$ is determined by the relation
$2g-2=E-W-B-F$. We refer to $g$ also as the genus of $\mathcal{D}$.

If $(a_j, b_k, d_l)=(a,b,c)$ for every choice of $j,k,l$, then  the dessin $\mathcal{D}$ is called \emph{uniform}. When $a=2$ the dessin is called \emph{clean}.
Uniform clean dessins will be particularly relevant in this work.
 
\

If we label the edges of $\mathcal{D}$ from $1$ to $E$, the orientation of $X$ induces  permutations $\sigma_0, \sigma_1 \in \mathbb{S}_E$ encoding the cyclic ordering of the different edges around the white  and black vertices, respectively.  The pair $(\sigma_0, \sigma_1)$ is called the \emph{permutation representation} of  $\mathcal{D}$, and the group $M= \langle \sigma_0, \sigma_1 \rangle$, which is transitive since $\mathcal{D}$ is connected, is called the \emph{monodromy group} of the dessin. Conversely, any pair of permutations $\sigma_0, \sigma_1$ generating a transitive subgroup of the symmetric group $\mathbb{S}_E$ arises as the permutation representation of a dessin d'enfant $\mathcal{D}$ with $E$ edges. The disjoint cycles of $\sigma_0$, $\sigma_1$ and $\sigma_{\infty}:=(\sigma_0 \sigma_1)^{-1}$ are in bijective correspondence with  the white vertices, black vertices and faces of $\mathcal{D}$, and their lengths encode its passport (see Figure \ref{fig_monodromy}). For instance, being uniform means that the permutations $\sigma_0, \sigma_1$ and $\sigma_{\infty}$ decompose as products of disjoint cycles of equal length.

\begin{figure}[!htbp]
\centering{
\includegraphics[width=0.32\textwidth]{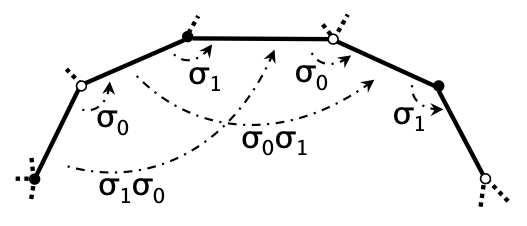} 
\caption{The monodromy of a dessin.}\label{fig_monodromy}
}
\end{figure}

As it is customary in the theory of topological surfaces, one can construct a planar model of $X$ by cutting it open appropriately along $\mathcal{D}$ to form a planar polygon $P$ endowed with the corresponding  identifications $\sim$ between pairs of sides in $\partial P$ that allow us to reconstruct the topology of $X$ as the quotient space $P/\sim$. This can be done in such a way that a face $C_i$ of $\mathcal{D}$ of degree $c_i$ corresponds to a topological $2c_i$-gon $P_i$ contained in $P$. Since $X=\bigcup_i C_i$ one has $P= \bigcup P_i$. The edges of $\mathcal{D}$  account for all the sides of the polygons $P_i$: an edge $e$ either lies in the interior of $P$ or on its boundary, and in the second case $e$ occurs twice in $\partial P$ by virtue of the side-pairing identifications. Of course, when $\mathcal{D}$ is a uniform dessin, all these polygons $P_i$ have the same number of edges.

\

\begin{example} \label{ej:M8g2plano} \rm{
Let $\sigma_0$ and $\sigma_1$ be the following permutations in  $\mathbb{S}_{16}$
$$
\begin{array}{c}
\sigma_0 =(1,16)(2,15)(3,14)(4,13)(5,12)(6,11)(7,10)(8,9)\\
  \sigma_1=(1,6,9,12)(2,10,16,8)(3,13,15,5)(4,7,14,11)
  \end{array}
$$

This is the monodromy pair of a dessin $\mathcal{D}\subset X$ with 16 edges, 8 white vertices of degree 2 and 4 black vertices of degree four. Since $\sigma_1 \sigma_0=(1,8,12,3,11,9,2,5)(4,15,10,14,13,7,16,6)$, $\mathcal{D}$ has two faces (both of degree 8), and it follows that the genus of $X$ is $g=2$. A planar model for the inclusion $\mathcal{D} \subset X$ can be obtained from two 16-gons by adding the labels provided by the cycles of $\sigma_1 \sigma_0$ to half the edges of each of the 16-gons, and using $\sigma_0$ to label the remaining edges. The result is depicted in the left hand side of Figure \ref{fig_ej_planog2}.

\begin{figure}[!htbp]
\centering{
\begin{tabular}{ccc}
\includegraphics[width=0.48\textwidth]{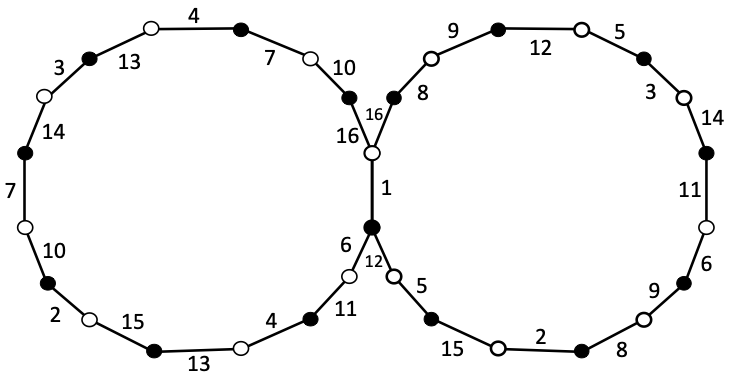} & \quad \quad  & \includegraphics[width=0.42\textwidth]{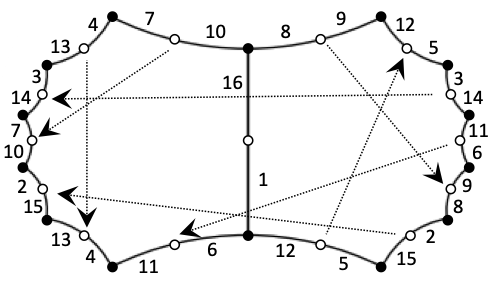} 
\end{tabular}
}
\caption{On the left side, a purely topological planar model of $\mathcal{D}\subset X$ constructed from the monodromy pair. On the right hand side, a topologically equivalent construction carried out in   hyperbolic space, determined by the Fuchsian representation of the same dessin.}\label{fig_ej_planog2}
\end{figure}
}
\end{example}

One of the key points of the theory of dessins d'enfants is that a dessin $\mathcal{D}\subset X$ determines the following data:
\begin{itemize}
\item A triangular decomposition of $X$ obtained by drawing a topological segment from each vertex (white or black) to a common center of the faces they belong, so that each edge of $\mathcal{D}$ becomes the side of two adjacent topological triangles which can be coherently provided with alternating labels $N$ and $S$.
\item A continuous surjection $p_{\mathcal{D}}: X \longrightarrow \widehat{\mathbb{C}}$ ramified over $\{0,1,\infty\}$ 
that maps each triangle labeled $N$ (resp. $S$) homeomorphically to the Northern (resp. Southern) Hemisphere in such a way that the inverse image 
$p_{\mathcal{D}}^{-1}([0,1])\subset X$ of the unit interval $[0,1]\subset \mathbb{C}$ is precisely the dessin $\mathcal{D}$ with white vertices corresponding to the points of $p_{\mathcal{D}}^{-1}(0)$, black vertices to $p_{\mathcal{D}}^{-1}(1)$ and face centers to $p_{\mathcal{D}}^{-1}(\infty)$.
\item A Riemann surface structure on $X$ provided by the atlas $\mathcal{A}_{\mathcal{D}}$ determined by using the local homeomorphism $p_{\mathcal{D}}$ to define local charts away the vertices and centers of faces of $\mathcal{D}$, so that $p_{\mathcal{D}}:(X, \mathcal{A}_{\mathcal{D}}) \longrightarrow \widehat{\mathbb{C}}$ becomes a covering of Riemann surfaces ramified only over $\{ 0, 1 , \infty \}$. Up to equivalence, this covering is independent of the choice of the triangulation associated to $\mathcal{D}$.
\end{itemize}

The so-called \emph{Belyi pair} $( (X, \mathcal{A}_{\mathcal{D}}) , p_{\mathcal{D}})$ corresponding to a dessin $\mathcal{D}$ of type $(a,b,c)$, which we shall denote more simply by $(S_{\mathcal{D}},\beta)$,  can be explicitly described in terms of Fuchsian groups. Let
$$
\Delta(a,b,c)= \langle x, y, z \  | \ x^a=y^b=z^c =xyz=1 \rangle < \mathrm{Aut}(\mathbb{D}) \simeq \mathbb{P}\mathrm{SL}(2, \mathbb{R})
$$
be a Fuchsian triangle group generated by three hyperbolic rotations $x,y$ and $z$ of angle $2\pi/a$, $2\pi/b$ and $2\pi/c$ respectively (elliptic isometries of orders $a,b$ and $c$), whose fixed points $w_a,w_b,w_c$ are the three vertices, ordered counterclockwise, of a hyperbolic triangle $T$ of angles $\pi/a, \pi/b, \pi/c$. The quadrilateral $Q$ given by the union of $T$ and its reflection along the side $[w_a, w_c]$ 
has vertices at $w_b, w_c, x(w_b)$ and $w_a$, with angles  $\pi/b, 2\pi/c, \pi/b$ and $2\pi/a $,  respectively. It  is a fundamental domain for $\Delta= \Delta(a,b,c)$. 
As $x$ maps the hyperbolic segment $[w_a, w_b]$ to $[w_a,x(w_b)]$ and $z$ maps $[w_c,x(w_b)]$ to $[w_c,w_b]$, the quotient space $\mathbb{D}/\Delta=Q/\Delta$ is isomorphic to the Riemann sphere $\widehat{\mathbb{C}}$, and we can realize the identification $\mathbb{D}/\Delta\simeq \widehat{\mathbb{C}}$  in such a way that the equivalence classes $[w_a]_{\Delta}, [w_b]_{\Delta},[w_c]_{\Delta} \in \mathbb{D}/\Delta$ correspond to $0, 1, \infty$ and, moreover, the hyperbolic segments $[w_a, w_b], [w_b, w_c]$ and $[w_c,w_a]$, lying in the boundary of $Q$, project to the segments $[0,1], [1, \infty], [-\infty,0]$ contained in $(\mathbb{R}\cup \{\infty \}) \subset \widehat{\mathbb{C}}$. We also observe that under this identification the Northern and Southern Hemispheres correspond to the triangle $T$ and its reflection.

\

The permutation representation of a dessin $\mathcal{D} \subset X$ of type $(a,b,c)$ and $E$ edges induces  a group homomorphism $\omega: \Delta(a,b,c) \longrightarrow \mathbb{S}_E$, called the \emph{monodromy homomorphism}, from the Fuchsian triangle group $\Delta(a,b,c)$ to the symmetric group in $E$ elements $\mathbb{S}_E$, 
determined by $\omega(x) = \sigma_0$, $\omega(y) = \sigma_1$ and $\omega(z)=\sigma_{\infty}$, whose image is the monodromy group $M$. It turns out that if $\mathrm{Stab}_M(1)$ is the stabiliser subgroup in $M$ of the edge labelled 1 and  $K=\omega^{-1}(\mathrm{Stab}_M(1))$ denotes its preimage, an index $E$ subgroup in $\Delta(a,b,c)$, then the quotient morphism  $\beta:S_{\mathcal{D}}=\mathbb{D}/K \longrightarrow \mathbb{D}/\Delta\simeq \widehat{\mathbb{C}}$  induced by the inclusion $K<\Delta(a,b,c)$ ramifies only over $[w_a]_{\Delta}\simeq 0, [w_b]_{\Delta}\simeq 1, [w_c]_{\Delta}\simeq \infty$ and is equivalent to the covering $p_{\mathcal{D}}: (X, \mathcal{A}_{\mathcal{D}}) \longrightarrow \widehat{\mathbb{C}}$; that is, there is an orientation-preserving homeomorphism $\varphi_{\mathcal{D}}: X \longrightarrow S_{\mathcal{D}}$ such that the following diagram commutes: 
$$
\xymatrix{ X \ar[rr]^{\varphi_{\mathcal{D}}} \ar[dr]_{p_{\mathcal{D}}} & & S_{\mathcal{D}} 
\ar[dl]^{\beta}
\\
 & \widehat{\mathbb{C}} & }
 $$

(Note that $\varphi_{\mathcal{D}}$ becomes automatically a isomorphism of Riemann surfaces when the topological surface $X$ is endowed with the Riemann surface structure provided by the dessin).

\

One can form a fundamental domain for $K$ as a union $Q_1\cup \ldots \cup Q_E$ of $\mathrm{deg}(\mathcal{D})= E$ copies of the quadrilateral $Q$. By construction, the  fibre of $\beta$ over the point $[w_c]_{\Delta}\simeq \infty$ consists of $F<E$ points $[w_1]_K, \ldots, [w_F]_K$ of branching orders $c_1, \ldots, c_F$ respectively. As a consequence  we can choose the quadrilaterals in such a way that $c_1$  of them have the same vertex $w_1$ of angle $2 \pi /c$, $c_2$ of them have the same vertex $w_2$ of angle $2 \pi /c$, etcetera. The inverse image of the unit interval $ [0,1] \simeq [w_a, w_b]\subset Q/\Delta = \mathbb{D}/\Delta$ consists of all the edges of the quadrilaterals $Q_1, \ldots , Q_E$ which  do not meet the points $w_1, \ldots, w_F$; and we make it a bicoloured graph by declaring the points in the fibres $\beta^{-1}(0)$ and $\beta^{-1}(1)$ as white and black points respectively. 

We observe that in view of the identification $\mathbb{D}/ \Delta \simeq \widehat{\mathbb{C}}$ made above, the commutativity of the last diagram implies that each of these hyperbolic quadrilaterals $Q_l$ corresponds via $\varphi_{\mathcal{D}}$ to a topological quadrilateral in $X$ formed by two adjacent topological triangles of the triangulation associated to $\mathcal{D}$. Likewise we
note that the union $P_j^*$ of the quadrilaterals meeting at a given point $w_j$ gives a hyperbolic version of one of the polygons which in the planar model described above corresponded to a 
 face; this is precisely the polygon $P_j$ that is mapped into $P_j^*$ via the homeomorphism $\varphi_{\mathcal{D}}$. 
 
Also by construction  the homeomorphism $\varphi_{\mathcal{D}}: X \longrightarrow S_{\mathcal{D}}$  restricts to an isomorphism of bicoloured graphs between $\mathcal{D}\subset X$ and $\beta^{-1}([0,1]) \subset S$. 

Finally, since, as mentioned above, the Belyi covering is well defined up to equivalence, a different choice of triangulation will result in a new homeomophism $\varphi'_{\mathcal{D}}: X \longrightarrow S_{\mathcal{D}}$ such that $\varphi'_{\mathcal{D}} \circ \varphi_{\mathcal{D}}^{-1} $ is an automorphism of $S_{\mathcal{D}}$.

 In other words, the dessin defines a point $[\varphi_{\mathcal{D}}, S_{\mathcal{D}}]\in T(X)$ which does not depend on the choices made to define the homeomorphism $\varphi_{\mathcal{D}}$.

\begin{definition} \rm
The pair $(S,\beta)$ and the Fuchsian group inclusion $K<\Delta(a,b,c)$ constructed above are called the \emph{Belyi pair} and the \emph{Fuchsian representation} of the dessin $\mathcal{D}\subset X$. The homeomorphism $\varphi_{\mathcal{D}}$ and the point  $[\varphi_{\mathcal{D}}, S_{\mathcal{D}}]\in T(X)$ will be referred to as the \emph{homeomorphism and the Grothendieck Teichm\"uller point associated} to $\mathcal{D}$.
\end{definition}

Note that the degree of $\mathcal{D}$, which is the number $E$ of edges, agrees with the degree of $\beta$ as a ramified covering of the sphere, and also with the index of the group inclusion $K<\Delta(a,b,c)$.

 \begin{example} \rm
 The picture at the right hand side of Figure \ref{fig_ej_planog2} contains a fundamental domain of the group $K$ that uniformizes the Belyi surface $S=\mathbb{D}/K$ associated to the uniform dessin $\mathcal{D}\subset X$  in Example \ref{ej:M8g2plano}. The identification of edges with the same label, that still determine the topology of the surface, is realized by certain edge-pairing transformations, which are materialized by suitable  isometries of the hyperbolic plane that generate a subgroup $K$ of index 16 in the Fuchsian triangle group $\Delta(2,4,8)$. The two pictures in  Figure \ref{fig_ej_planog2} are topologically equivalent to each other. The homeomorphism $\varphi_{\mathcal{D}}: X \longrightarrow S_{\mathcal{D}}$  restricts to an isomorphism of the bicoloured graphs at both sides of the figure.
 \end{example}
 
 The significance of dessin d'enfants rests on the following result:
 
 \begin{theorem*}[Belyi \cite{Belyi_1979} - Grothendieck \cite{Grothendieck_1997}]
 A Riemann surface is defined over a number field if and only if it is the Riemann surface associated to a dessin.
 \end{theorem*}

\subsection{The Fuchsian representation of a uniform  dessin} \label{sec:hypuniclean}

\begin{lemma}
Let $\mathcal{D} \subset X$ be a uniform  dessin,  $(S, \beta)$ the associated Belyi pair and $K<\Delta(a,b,c)$ its Fuchsian group representation. Then:
\begin{enumerate}
\item $K$ is torsion free.
\item The edges of $\beta^{-1}([0,1]) = \varphi_{\mathcal{D}}(\mathcal{D}) \subset S$ are geodesic segments with respect to the hyperbolic metric in $S$. 
\end{enumerate}
\end{lemma}

\begin{proof}
1) This is a well-known fact. The point is that all torsion elements $t\in \Delta(a,b,c)$ are powers of conjugates of the elliptic generators $x,y$ and $z$ and therefore the corresponding permutations $\omega(t)$ are conjugates of powers of $\sigma_0, \sigma_1$ and $\sigma_{\infty}$. But if $\mathcal{D}$ is uniform these permutations decompose as disjoint products of cycles of equal length, hence they will never fix an edge and so $t$ cannot lie in  $K=\omega^{-1}(\mathrm{Stab}_M(1))$.

2) The fundamental domain for $K$ decomposes as a union of  hyperbolic $2c$-gons with angles $2\pi / a$ and  $2\pi / b$ (alternating). Since by 1) the natural quotient map $\mathbb{D} \longrightarrow \mathbb{D}/K$ is in this case a local isometry, the sides of the $2c$-gons, which after the required identification in pairs become the counterpart of the edges of $\mathcal{D}$ in  $S=\mathbb{D}/K$, are certainly geodesic segments for the hyperbolic metric in $S$.
 \end{proof}

Recall for further use that a uniform dessin $\mathcal{D}$ of type $(a,b,c)$ is called \emph{regular} if the Belyi function $\beta$ induced by the inclusion of $K$ in $\Delta=\Delta(a,b,c)$ is a normal covering of the Riemann sphere $ \widehat{\mathbb{C}} \simeq \mathbb{D}/\Delta(a,b,c)$.  This is equivalent to saying that $K \lhd \Delta(a,b,c)$ and in this case the Belyi function agrees with the quotient map of $S=\mathbb{D} / K$ under the action of the group of isometries $\Delta/K$. The monodromy group of a regular dessin can be identified to $\Delta/K$ and so its order coincides with the degree of the dessin.


\section{Filling curves and dessins d'enfants}

Let $\gamma \subset X$ be a smooth closed curve (with respect to any differentiable structure on $X$) and let $p\in X$ be a self-intersection point. If $m_p$ is the maximal number of branches of $\gamma$ that meet at $p$ with the property that they are all mutually transversal, then one says that ${m_p}\choose{2}$ is the \emph{number of transversal} (tangent vector) \emph{pairs} of $\gamma$ at $p$. The total number of transversal pairs $\sum_p {{m_p}\choose{2}}$ is called the \emph{self-intersection number} of $\gamma$. A curve $\gamma$ of a closed oriented topological surface $X$ is said to be \emph{in minimal position} if its self-intersection number is minimal among all curves (freely) homotopic to $\gamma$. An equivalent condition is the absence of \emph{monogons} and \emph{bigons}, i.e. connected components of $X\smallsetminus \Gamma$ homeomorphic to discs such that the boundary is a a single arc connecting a given self-intersection point of $\Gamma$ to itself or two arcs connecting a same pair of self-intersection points, see \cite{Farb_Margalit_2012} and \cite{Hass_Scott_1985}. It is known that the geodesic representative (with respect to any conformal structure in $X$) of the homotopy class of $\gamma$ is in minimal position. 

We will say that a homotopically non-trivial curve $\gamma \subset X$ is  \emph{in general position} if it is in minimal position and
$m_p=2$ at every self-intersection point $p$. The self-intersection number of a curve $\gamma$ in general position agrees with the number of its self-intersection points. See \cite{Basmajian_2013}. 

These concepts readily extend to the case of a pair of homotopically nontrivial and distinct curves $\gamma_1$, $\gamma_2$. This allows us to speak about the self-intersection number of $\gamma_1, \gamma_2$ or whether they are in minimal or general position. We will say that a multicurve $\Gamma=\{ \gamma_1, \ldots , \gamma_r \}$ in $X$ is in minimal (resp. general) position if each individual curve $\gamma_i$ and each pair $\gamma_i, \gamma_j$ with $i \neq j$ is in minimal (resp. general) position.

 One says that $\gamma$ is a \emph{filling curve} if every representative of its homotopy class intersects every homotopically non-trivial loop on $X$, and this same condition defines the notion of \emph{filling multicurves}. Thus, a filling (multi)curve has complementary components homeomorphic to discs.

 \
 
 One can easily make a filling multicurve $\Gamma$ in minimal position into a bicoloured graph $\mathcal{D}_{\Gamma}$ simply by declaring  
 the set $\mathcal{B}$ of self-intersection points as the set of black vertices of $\mathcal{D}_{\Gamma}$, and choosing a middle point in each of the $d$ connected components
  of $\Gamma \smallsetminus \mathcal{B}$ to form the set of white vertices. We shall reserve the word \emph{arc} to refer to these connected components. There are therefore $d$ arcs in $\Gamma$ and  $E=2d$ edges in $\mathcal{D}_{\Gamma}$. In other words,  $\mathcal{D}_{\Gamma} \subset X$ is a dessin d'enfant of degree $2d$ and type $(2,2m,k)$, where $2m$ and $k$ are respectively the least common multiples of the degrees of black vertices and faces. Note that the degree of a black vertex $p$ being $2m_p$ means the number of transversal pairs at the self-intersection point $p$ is ${m_p}\choose{2}$. We observe 
that $\mathcal{D}_{\Gamma} $ is a  clean dessin by construction. We will be particularly interested in the case in which, in addition, $\mathcal{D}_{\Gamma} $ is  uniform. We note that this condition already forces $\Gamma$ to be in minimal position, for the occurrence of monogons or bigons would imply $k\le 2$ and so $(2,2m,k)$ would not be a hyperbolic type, that is, $X$ could not be of genus $g\ge 2$.
 
\

The following result shows that every uniform dessin $\mathcal{D}$ of type $(2,2m,k)$ arises in the way described above. 

\begin{proposition} \label{pr:destomulticurve}
Let  $\mathcal{D}\subset X$ be a clean  uniform dessin d'enfant of type $(2,2m,k)$ and genus $g\ge 2$ with permutation representation $\sigma_0, \sigma_1$. Let $\beta: S=\mathbb{D} / K \longrightarrow \mathbb{D}/ \Delta$ be the corresponding Belyi pair and $\varphi_{\mathcal{D}}: X \longrightarrow S$ the associated homeomorphism.  Then:
\begin{enumerate}
\item There is a filling multicurve $\Gamma'= \{\gamma'_1, \ldots, \gamma'_r \}\subset S$ such that $\varphi_{\mathcal{D}}(\mathcal{D})=\beta^{-1}([0,1])= \mathcal{D}_{\Gamma'} \subset S$. Moreover the curves $\gamma'_i$ are geodesics of $S$ and the number $r$ of components of $\Gamma'$ agrees with half the number of cycles of $\sigma_1^m \sigma_0$.
\item If $\mathcal{D}=\mathcal{D}_{\Gamma}$ for some filling multicurve  $  \Gamma =  \{\gamma_1, \ldots, \gamma_r \} \subset X$, then $\Gamma'=\varphi_{\mathcal{D}}(\Gamma)$ (as multicurves, that is $\gamma'_i= \varphi_{\mathcal{D}} ( \gamma_i ) $).
\end{enumerate}
\end{proposition}

\begin{proof} 
As explained in Section \ref{sec:review_dessins}, we can consider a fundamental domain for $K$ in $\mathbb{D}$ given as a union of a certain number of hyperbolic $2k$-gons of angles $2\pi/2=\pi$ and $2\pi/2m$, that is to say regular $k$-gons of angle  $\pi/m$.
The number $n$ of polygons and the number $k$ are related to the degree $E$ of $\mathcal{D}$ by the relation $nk=E$. The edges of  $\varphi_{\mathcal{D}}(\mathcal{D})$ are in this case geodesic arcs in $S$.

Let $e_1$ and  $e_2$ be the pair of edges of $\varphi_{\mathcal{D}}(\mathcal{D})$ meeting (with angle $\pi$) at a given white vertex $w$, and let  $b_1$ and $b_2$ be the black vertices of $e_1$ and $e_2$ respectively (it may happen that $b_1=b_2$).  The union $e_1\cup e_2$ corresponds to  a full side $e$ of one of the regular $k$-gons that form the fundamental domain for $K$. Let $\gamma=\gamma^{e}$ be the full geodesic of $S$ that contains $e_1$ (and $e_2$). Clearly $(\sigma_0, \sigma_1)$ is also the permutation representation pair of $\varphi_{\mathcal{D}}(\mathcal{D})$. Obviously 
 $\sigma_0(e_1)=e_2$, and therefore $\sigma_1^m\sigma_0(e_1)$ is a geodesic segment meeting $e_2$ at $b_2$ with an angle $m \frac{\pi}{m}= \pi$. The sequence $e_1, e_2,  \sigma_1^m\sigma_0(e_1), \sigma_0(\sigma_1^m\sigma_0(e_1)), (\sigma_1^m\sigma_0)^2(e_1), \ldots$ describes the order we encounter the sides of the dessin $\varphi_{\mathcal{D}}(\mathcal{D})$ as we traverse the geodesic starting from the directed segment $e_1=[b_1,w]$. After a certain number of steps a full cycle $\tau_{e_1}$ of $\sigma_1^m \sigma_0$ is completed, and accordingly the geodesic path  returns to $e_1$, which shows that $\gamma$ is a closed geodesic. Obviously, companion to $\tau_{e_1}$ there is another cycle $\tau_{e_2}$ of same length containing the edge $e_2$. The union of all the edges of $\varphi_{\mathcal{D}}(\mathcal{D})$  that appear in $\tau_{e_1}$ and $\tau_{e_2}$ account for the whole closed geodesic $\gamma$.

If $\gamma^e$ covers $\varphi_{\mathcal{D}}(\mathcal{D})$ completely then $\varphi_{\mathcal{D}}(\mathcal{D}) =\mathcal{D}_{\Gamma'}$ for $\Gamma'=\{\gamma^e\}$. This happens when  $ \sigma_1^m\sigma_0$ decomposes as a product of two disjoint cycles of equal length (thus, of length $d=nk/2$). If not, $\gamma^e=\gamma'_1$ will  just be one of the components of a family $\Gamma'=\{ \gamma'_1, \ldots, \gamma'_r \}$ of closed curves that fills up $S$. The  second component can be described in a similar way starting from an edge that does not belong to $\gamma'_1$, and so on. The number $r$ of components of this family $\Gamma'$  is clearly half the number of disjoint cycles of $ \sigma_1^m\sigma_0$.

Finally, let us assume that $\mathcal{D}=\mathcal{D}_{\Gamma}$ for some filling multicurve $\Gamma \subset X$. We must show that in this case the way we have traveled the graph $\beta^{-1}([0,1])$ is the same as the way the multicurve $\Gamma'= \varphi_{\mathcal{D}}(\Gamma)$ is traversed (or the opposite). To see this, we only notice that in this situation when in the discussion obove we arrive at the endpoint of $e_2$, that is, at the transversal self-intersection point $b_2$, the fact that $\varphi_{\mathcal{D}}(\Gamma)$ is in minimal position implies that the only possible way to continue our walk is by taking the opposite side of $\varphi_{\mathcal{D}} (\mathcal{D}_{\Gamma} )$, that is, the side $\sigma_1^m(e_2)=\sigma_1^m \sigma_0(e_1)$, and this is exactly the choice we made.
\end{proof}

From the above proof we can extract the following consequence:

\begin{corollary}
A clean  uniform dessin of type $(2,2m,k)$ with monodromy pair $(\sigma_0, \sigma_1)$ determines a filling curve if and only if $\sigma_1^m \sigma_0$ decomposes as a product of two disjoint cycles of the same length.
\end{corollary}

We warn the reader that, by abuse of notation, in the sequel sometimes we will not distinguish the multicurve $\Gamma$ from the associated dessin $\mathcal{D}_{\Gamma}$, and so we will denote the homeomorphism $\varphi_{\mathcal{D}_{\Gamma}}: X \longrightarrow S_{\mathcal{D}_{\Gamma}}$ simply by $\varphi_{\Gamma}: X \longrightarrow S_{\Gamma}$.

\begin{example} \label{ex:dessinexample} \rm 
The filling curve $\gamma$ in the left hand side of Figure \ref{fig_g2a} is in minimal position. The four self-intersection points are going to account for the set $\mathcal{B}$ of black vertices of the associated dessin. The complement $\gamma \smallsetminus \mathcal{B}$ consists of eight (directed) arcs to which we give a label $j$ (from 1 to 8) as the numbered arrows indicate. Placing a white vertex $w_j$ in the middle splits each arc $j$ into two directed arcs: the first one ending at $w_j$, which we still label $j$, and the second one emanating from $w_j$, to which we give the complementary label  label $17-j$. In this way the monodromy representation of our dessin $\mathcal{D}_{\gamma}$ is given by the permutations 
$$
\begin{array}{c}
\sigma_0 =(1,16)(2,15)(3,14)(4,13)(5,12)(6,11)(7,10)(8,9)\\
  \sigma_1=(1,6,9,12)(2,10,16,8)(3,13,15,5)(4,7,14,11)
 \end{array}
$$ 

Thus, $\mathcal{D}_{\gamma}$ is nothing but the uniform dessin of type $(2,4,8)$ we introduced in Example \ref{ej:M8g2plano}. Notice that in this case $\sigma_1^2\sigma_0=(1,2,3,4,5,6,7,8)(9,10,11,12,13,14,15,16)$ decomposes as a product of just two cycles, as expected for the case of a filling curve (rather than a multicurve with $r>1$).
\end{example} 

\section{Explicit minima of the geodesic length function} \label{sec_main}

We see in the proof of Proposition \ref{pr:destomulticurve} that the length of the multicurve $\Gamma' \subset S $ associated to a clean uniform dessin d'enfant of genus $g\ge 2$ depends solely on the type of the dessin and not on its homotopy class or the number of its components.
This sum equals  $d=nk/2$ times the length of the side of a regular hyperbolic $k$-gon of angle $\pi/m$.  Elementary hyperbolic geometry  shows that this sum of lengths equals
\begin{equation} \label{eq:l_2m,k,d}
\ell_{2m,k,d}= d \cdot  \cosh^{-1} \left(\displaystyle\frac{\cos^2\frac{\pi}{2m} + \cos \frac{2\pi}{k}}{\sin^2\frac{\pi}{2m}} \right)
\end{equation}
which in terms of the hyperbolic arc $L_{m,k}$ introduced in Section \ref{subsec:intro} can be rewritten as 
$$2d \cdot (\mbox{length of }L_{m,k})=2d \cdot  \cosh^{-1} \left(\displaystyle\frac{\cos \frac{\pi}{k}}{\sin\frac{\pi}{2m}} \right).
$$

The main result of this paper is the following:

\begin{theorem} \label{th:main}
Let $X$ be a closed oriented topological surface of genus $g\ge2$, and let $\Gamma \subset X$ be a filling multicurve. 
Assume that the underlying dessin d'enfant  $\mathcal{D}_{\Gamma}$ is uniform of type $(2,2m,k)$. Then the geodesic length function $\ell_{\Gamma}$ reaches its minimum precisely at the Grothendieck Teichm\"uller point $[\varphi_{\Gamma}, S_{\Gamma}] \in T(X)$.
 Moreover, if we denote by $\beta: S_{\Gamma} \to \widehat{\mathbb{C}}$ the corresponding Belyi function, the geodesic multicurve in $S_{\Gamma}$ representing $\varphi_{\Gamma}(\Gamma)$ has $\beta^{-1}([0,1])$ as underlying graph and its length equals  the quantity 
$\ell_{2m,k,d}$ displayed above, where $d$ is the number of arcs of $\Gamma$.
\end{theorem}

\begin{proof}
Thorough the identification $T(X)\simeq T(S_{\Gamma})$ induced by the homeomorphism $\varphi_{\Gamma}: X \longrightarrow S_{\Gamma}$ (given by $[f, Y] \longmapsto [f \circ \varphi_{\Gamma}^{-1}, Y]$), the point $[\varphi_{\Gamma}, S_{\Gamma}]\in T(X) $ corresponds to the point $[\mathrm{Id}, S_{\Gamma}] \in T(S_{\Gamma})$. Thus, what we need to prove is that, if we set $\varphi_{\Gamma}(\Gamma)=\Gamma'$, 
the function $\ell_{\Gamma'}: T(S_{\Gamma}) \longrightarrow \mathbb{R}$ reaches its minimum at the point $[\mathrm{Id}, S_{\Gamma}]\in T(S_{\Gamma}) $.

\

As a first step, let us assume that $\mathcal{D}_{\Gamma}$ is a regular dessin. Then the Belyi pair $(S_{\Gamma}, \beta)$ is equivalent to the quotient map 
$$
\beta: S_{\Gamma}=\mathbb{D}/K \longrightarrow S_{\Gamma}/H=\mathbb{D}/\Delta(2,2m,k)=\widehat{\mathbb{C}}
$$
where $H<\mathrm{Isom}^+(S_{\Gamma})$ is isomorphic to the factor group   $\Delta(2,2m,k)/ K$. 

\

Since setwise $\Gamma'$ agrees with $\beta^{-1}([0,1])$ one has 
$$
h (\Gamma') = \Gamma' \quad \mbox{ for every } h \in H
$$
which obviously implies the following identity of geodesic length functions:
$$
\ell_{\Gamma'}= \ell_{h(\Gamma')}: T_g = T(S_{\Gamma}) \longrightarrow \mathbb{R}^+, \quad \mbox{ for every } h \in H
$$

Thus, for every $[f,S]\in T_g$ and for every  $h \in H$ the following identity holds:
$$
\ell_{\Gamma'}([f,S])= \ell_{h(\Gamma')}([f,S])= \ell_{\Gamma'}([ f \circ h, S]).
$$

Now, suppose that $[f,S]\in T_g$ reaches the minimum of the function $\ell_{\Gamma'}$. Then, since this minimum is unique, we must have
$$
[f,S]=[f \circ h, S] \quad \forall h \in H.
$$

Viewing $H$ as a subgroup of $\mathrm{Mod}_g=\mathrm{Mod}(S_{\Gamma})$ this means that $[f,S]\in \mathrm{Fix}(H)$, the fixed point set of $H$ in $T_g$. Obviously, $[ \mathrm{Id}, S_{\Gamma}]\in \mathrm{Fix}(H)$ too and, since $\mathrm{Fix}(H)$ can be identified to the Teichm\"uller space $T_{0,3}$ of the orbifold $S/H=(\widehat{\mathbb{C}}, \{0, 1, \infty\})$ which consists of a single point (see e.g. \cite{Gonzalez-Diez_Harvey_1992, Harvey_1971, Kravetz_1959}), we conclude that $[f,S]=[\mathrm{Id}, S_{\Gamma}]$. As the expression for the length of $\beta^{-1}[0,1]$  in $S_{\Gamma}$ was given in (\ref{eq:l_2m,k,d}), we are done.

\

Suppose now that $\mathcal{D}_{\Gamma}$ is not regular but merely uniform. If $\widetilde{\beta}: \widetilde{S_{\Gamma}} \longrightarrow \mathbb{C}$ is the normalization of the Belyi morphism $\beta: S_{\Gamma} \longrightarrow  \mathbb{C}$ we have a diagram as follows
$$
\xymatrix{
\widetilde{S_{\Gamma}} = \mathbb{D}/\widetilde{K}  \ar[r]^{p}    \ar[dr]^{\widetilde{\beta}} & 
S_{\Gamma}=\mathbb{D}/K \ar[d]_{\beta}  \\
& \widetilde{S_{\Gamma}}/\widetilde{H}=\widehat{\mathbb{C}} =\mathbb{D}/\Delta(2,2m,k) 
}
$$
where $ \widetilde{K}$ is the core subgroup of the inclusion $K<\Delta$, defined as 
$$
\widetilde{K}= \displaystyle\bigcap_{\gamma \in \Delta(2,2m,k)}\gamma K \gamma^{-1}\vartriangleleft \Delta(2,2m,k),
$$ 
and  $\widetilde{H} = \Delta(2,2m,k)/ \widetilde{K}$.  Note that $\widetilde{S_{\Gamma}} = S_{\widetilde{\Gamma}}$, where  $\widetilde{\Gamma}$ is the dessin corresponding to the regular Belyi pair $(\widetilde{S_{\Gamma}} , \widetilde{\beta})$, that is $\widetilde{\Gamma}=\widetilde{\beta}^{-1}([0,1])=(\beta \circ p)^{-1}([0,1])=p^{-1}(\Gamma')$.

\

As $\widetilde{\beta}$ is a regular covering, we already know that the minimum of $\ell_{\widetilde{\Gamma}}$ is $\ell_{\widetilde{\Gamma}}([\mathrm{Id}, S_{\widetilde{\Gamma}}])= \mathrm{deg}(p)\ell_{\Gamma'}([\mathrm{Id}, S_{\Gamma} ])$. We claim that from this fact it follows that the minimum of $\ell_{\Gamma'}$ must be $\ell_{\Gamma'} ([\mathrm{Id},  S_{ \Gamma }])$.

\

Indeed, if the minimum of $\ell_{\Gamma'}$ were attained at $[f,S] \neq [\mbox{Id}, S_{\Gamma}]$ then by considering a lift $\widetilde{f}: \widetilde{S_{\Gamma}} \longrightarrow \widetilde{S}$ of $f: S_{\Gamma} \longrightarrow S$ we would have 
$$
\ell_{\widetilde{\Gamma}}([\mathrm{Id}, S_{\widetilde{\Gamma}}]) = \mathrm{deg}(p) \ell_{\Gamma'}([\mathrm{Id}, S_{\Gamma}])>\mathrm{deg}(p)\ell_{\Gamma'}([f,S])= \ell_{\widetilde{\Gamma}}([ \widetilde{f}, \widetilde{S}]),
$$
a contradiction.
\end{proof}

Theorem \ref{th:main} identifies the Riemann surface at which the geodesic length function $\ell_{\Gamma}$ attains its minimum for all filling multicurves that admit a homotopy equivalent representative with the property that all its self-intersection points have the same degree $2m$ (or, equivalently, the same self-intersection number  ${m}\choose{2}$) and all the faces of its complement have the same number of edges. Only for those with $m=2$ the minimizing geodesic is in general position. For instance, for the curve in Figure \ref{fig_y2=x6-1} this is not the case.

\

A combination of Theorem \ref{th:main} and Proposition \ref{pr:destomulticurve} gives the following:

\begin{theorem} \label{th:comb}
Any genus $g$ surface subgroup $K<\Delta(2,2m,k)$ uniformizes a Riemann surface $S=\mathbb{D}/K$ which realizes the minimum of the geodesic length function $\ell_{\Gamma}$ of some filling multicurve $\Gamma$. In fact, if $\beta: \mathbb{D}/K \longrightarrow \mathbb{D} / \Delta = \widehat{\mathbb{C}}$ denotes the projection induced by the inclusion $K<\Delta$, 
 the multicurve $\Gamma$ can be taken to be $\beta^{-1}([0,1])\subset S$ and $(S, \beta)$ will be the Belyi pair associated to the dessin $\mathcal{D}_{\Gamma}$.

Moreover, if $\sigma_0$, $\sigma_1$ are the two permutations representing $\mathcal{D}_{\Gamma}$, the number $r$ of components of $\Gamma$ coincides with half of the number of disjoint cycles of $\sigma_1^m \sigma_0$.
\end{theorem}

\begin{example} \rm
From what has been worked out in the previous examples, we know that for the filling curve $\gamma$ at the left hand side of Figure \ref{fig_g2a}, the minimum of $\ell_{\gamma}$ is attained at the compact Riemann surface $S$ uniformized by the surface subgroup $K$ of the triangle group $\Delta(2,4,8)=\langle x, y, z \ | \ x^2=y^4=z^8=xyz=1 \rangle $ generated by the side pairing transformations that identify edges with equal labels in the fundamental domain on the right hand side of Figure \ref{fig_ej_planog2}. It is fairly simple to describe these side-pairing transformations in terms of the generators $x,y,z$ of $\Delta(2,4,8)$ taken as hyperbolic rotations of respective positive angles $\pi, \pi/2 $ and $\pi/4$ fixing the white point of edge 1, the black point of edge 1 and the center of the face on the right.  This is done in Table \ref{ta:gen}.

\

\begin{table}[!htbp]
\centering{ 
\begin{tabular}{|c|c|c|c|c|c|c|c|}
\hline 
Sides  & 2,15 & 3,14 & 4,13 & 5,12 & 6,11 & 7,10 & 8,9  \\ \hline
Word in $x,y,z$ &  $xz^5xz^6$ &   $xz^3xz^3$ & $xz^6xz^6x$ &   $z^6xz^7$ &    $xz^7xz^4$ &    $xz^4xz^7x$  &  $z^3xz$   \\ \hline
\end{tabular}}  \caption{Generators of the group $K$ described on the right hand side of Figure \ref{fig_ej_planog2} in terms of the generators $x,y,z$ of $\Delta(2,4,8)$.}  \label{ta:gen}
\end{table}

 The reader can also easily check that the elements of $\Delta(2,4,8)$ that occur in Table \ref{ta:gen} 
 do indeed map into  the stabilizer of $1\in \{1, \ldots , 16\}$ under the monodromy homomorphism  $\omega: \Delta(2,4,8) \longrightarrow \mathbb{S}_{16}$ in agreement with the fact, explained in Section \ref{sec:review_dessins}, that $K=\omega^{-1}(\mathrm{Stab}_M(1))$.
 
 We notice that this filling curve $\gamma$ which we are analyzing through Examples \ref{ej:M8g2plano} to \ref{ex:dessinexample}, agrees with the curve $\gamma_0$ drawn in Figure 6 of  \cite{Leininger_2003}, as can be seen by checking that the permutations $\sigma_0, \sigma_1$ representing the associated dessin $\mathcal{D}_{\gamma_0}$ are conjugate (or, in fact, agree for a suitable choice of the labels of the edges) to the permutation representation of $\mathcal{D}_{\gamma}$ given in Example \ref{ex:dessinexample}. Incidentally, it can be shown that there are another 18 uniform multicurves of genus 2 and type $(2,4,8)$, three of which are proper filling curves as the one we have considered here.
\end{example}

\begin{example} \label{ex:(2,4,12)gen2} \rm
It is known that there are exactly six uniform dessins of genus $g=2$ and type $(2,4,12)$ (cf \cite{Girondo_Torres_2011}). From the point of view of filling curves this means that, up to homotopy
equivalence, in genus $2$ there are exactly 6 filling multicurves with self-intersection number $3$ (as any representative with 3 self-intersection points in general position must have only one face, hence  corresponds to a dessin of type $(2,4,12)$). 
Figure   \ref{fig_2_4_12_gen2} shows a topological picture of these multicurves $\Gamma_i$ together with the hyperbolic Riemann surfaces (described as a hyperbolic 12-gon with the indicated side pairing) at which the  geodesic length functions $\ell_{\Gamma_i}$ reach their minimum.

 \begin{figure}[!htbp]
\centering{
\begin{tabular}{|c|c|} \hline
\begin{tabular}{c} \includegraphics[width=0.32\textwidth]{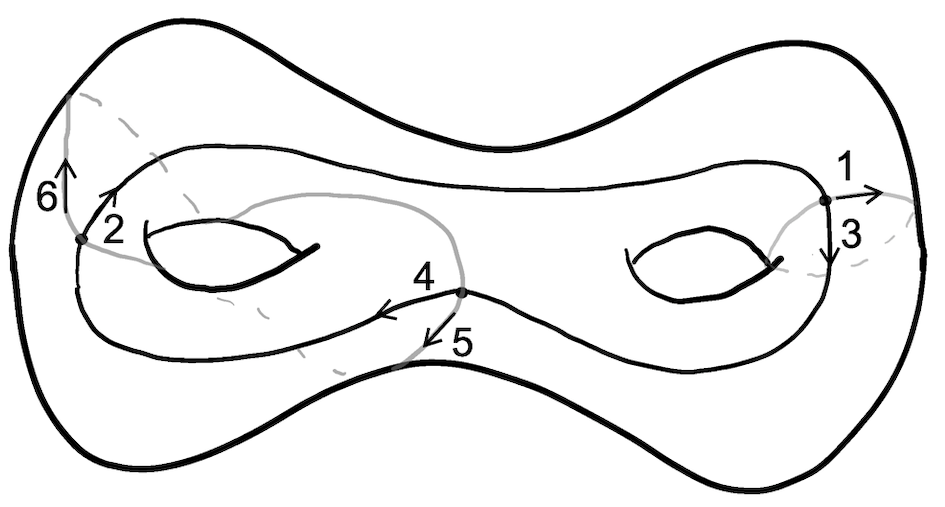} \end{tabular} &  \begin{tabular}{c} \includegraphics[width=0.32\textwidth]{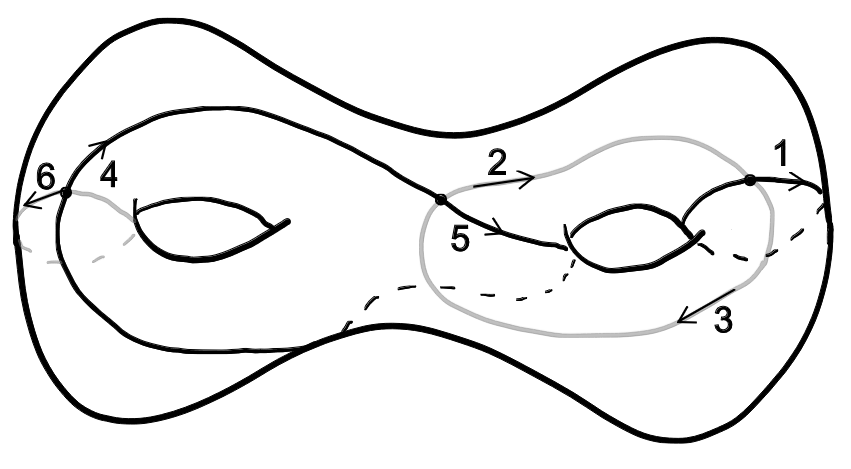} \end{tabular} \\
\begin{tabular}{c} \includegraphics[width=0.20\textwidth]{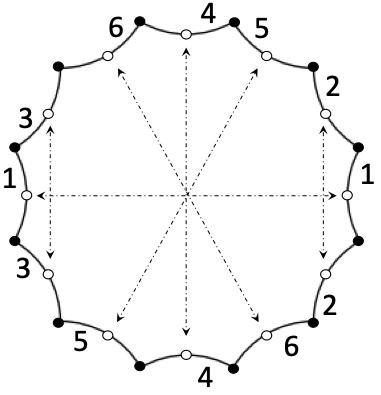} \end{tabular}  & \begin{tabular}{c} \includegraphics[width=0.20\textwidth]{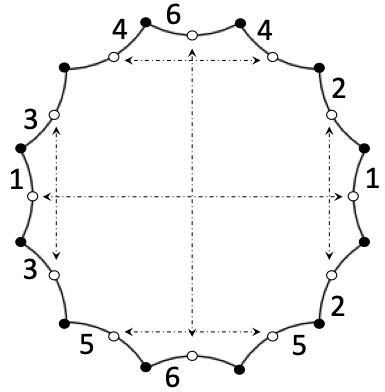} \end{tabular} \\ \hline 
\begin{tabular}{c} \includegraphics[width=0.32\textwidth]{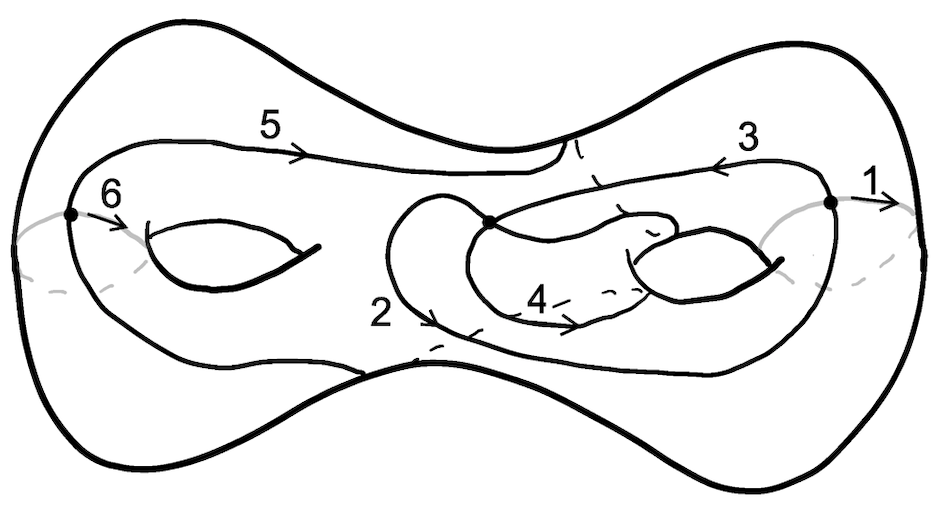} \end{tabular} & \begin{tabular}{c} \includegraphics[width=0.32\textwidth]{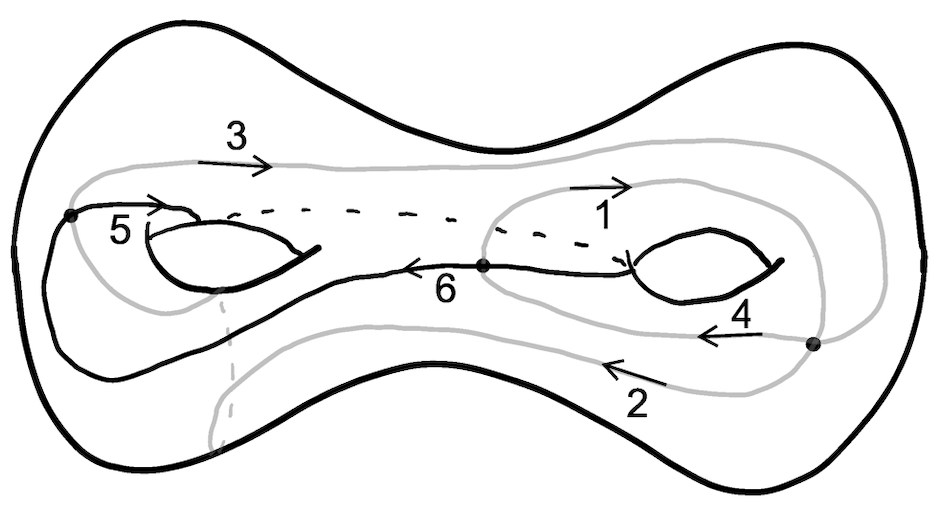} \end{tabular} \\
\begin{tabular}{c} \includegraphics[width=0.20\textwidth]{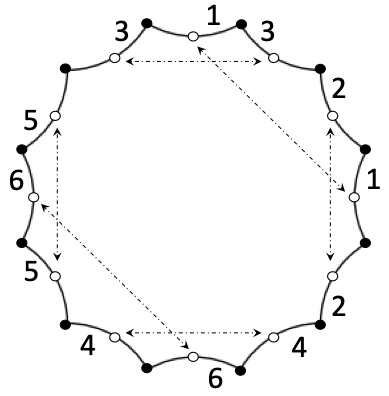} \end{tabular} & \begin{tabular}{c} \includegraphics[width=0.20\textwidth]{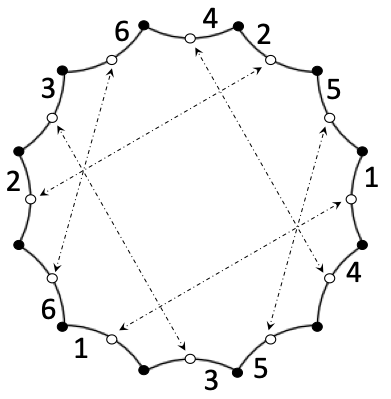} \end{tabular} \\ \hline
\begin{tabular}{c} \includegraphics[width=0.32\textwidth]{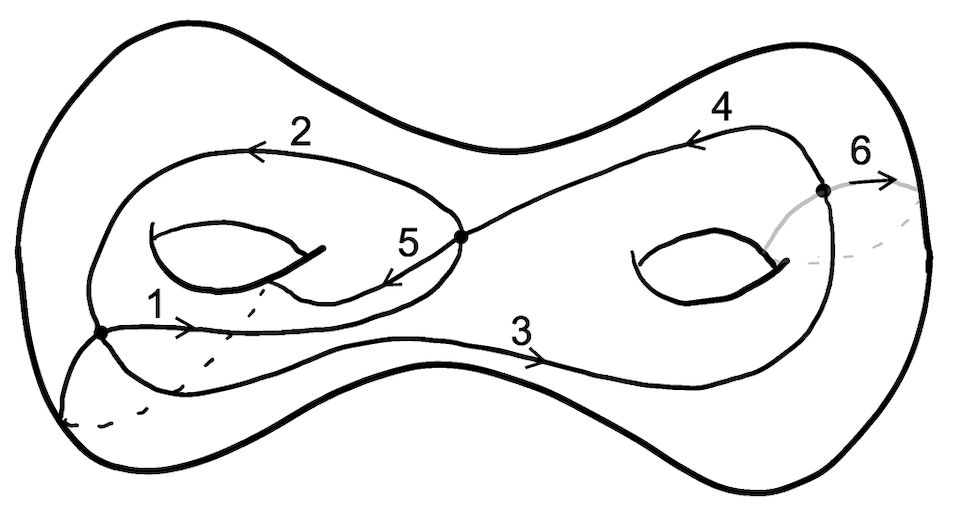} \end{tabular} &  \begin{tabular}{c} \includegraphics[width=0.32\textwidth]{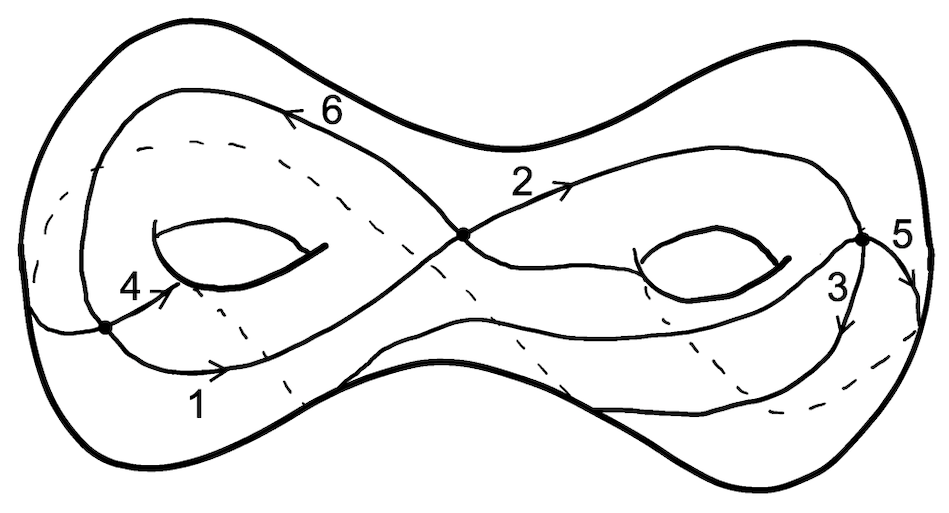} \end{tabular} \\
\begin{tabular}{c} \includegraphics[width=0.20\textwidth]{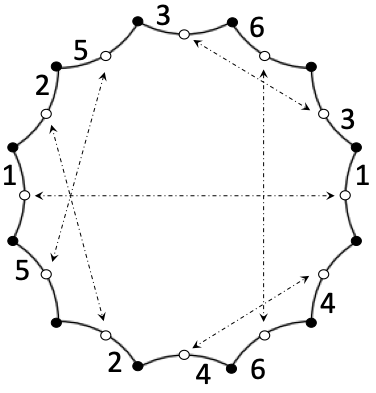} \end{tabular} & \begin{tabular}{c} \includegraphics[width=0.20\textwidth]{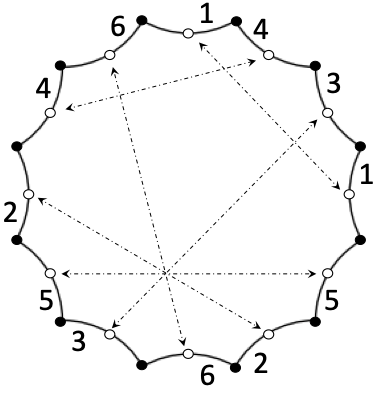} \end{tabular} \\ \hline
\end{tabular}
} \caption{The six uniform multicurves $\Gamma_i$ of type $(2,4,12)$ in genus $g=2$ and the corresponding hyperbolic surfaces at which each length function $\ell_{\Gamma_i}$ reaches its minimum.}\label{fig_2_4_12_gen2}
\end{figure}

To obtain the permutation representation of these dessins, we have labelled the edges $1$ to $6$ as is shown in the figure, and then added  labels $7$ to $12$ in such a way that the permutation $\sigma_0$ is in every case $(1,12)(2,11)(3,10)(4,9)(5,8)(6,7)$. Then the permutation $\sigma_1$ and the product $\sigma_1^2 \sigma_0$ are in each case as it is shown in Table \ref{ta:list}.

\begin{table}[!htbp]
\centering{ 
\begin{tabular}{|c|c|}
\hline
$\sigma_1$ & $\sigma_1^2 \sigma_0$ \\ 
\hline
\begin{tabular}{c}
$(1,3,12,11)(2,6,9,8)(4,5,10,7)$ \\
$(1,3,12,11)(2,5,10,9)(4,5,10,7)$\\
$(1,3,12,11)(2,4,8,10)(5,7,9,6)$\\
$(1,6,9,8)(2,10,12,4)(3,7,11,5)$\\
$(1,11,8,3)(2,5,12,9)(4,7,10,6) $\\
$(1,4,7,10)(2,6,12,8)(3,5,11,9)$
\end{tabular}& 
 \begin{tabular}{c}
(1)(2,3,4)(5,6)(7,8)(9,10,11)(12) \\
(1)(2,3)(4,5)(6)(7)(8,9)(10,11)(12) \\
(1)(2,3,4,5)(6)(7)(8,9,10,11)(12) \\
(1,2,3,4)(5,6)(7,8)(9,10,11,12)\\ 
(1,2,3,4,5)(6)(7)(8,9,10,11,12)\\ 
(1,2,3,4,5,6)(7,8,9,10,11,12) 
 \end{tabular} \\
 \hline
\end{tabular}
}
\caption{The monodromy representation of the six dessins determined by uniform multicurves of genus 2 and type $(2,4,12)$ are given by  $\sigma_0=(1,12)(2,11)(3,10)(4,9)(5,8)(6,7)$ and the permutation $\sigma_1$ in this list. The cycle structure of $\sigma_1^2\sigma_0$ is related to the components of the multicurve.}  \label{ta:list}
\end{table}

\

 It is interesting to notice that despite of the fact that the sum of the lengths of the curves in $\Gamma_j$ is the same for all $j$, the set of individual lengths of the curves is different for each of the six collections. Note also that only $\Gamma_6$ has just one component, i.e. it is the only element in this family of dessins that produces a filling curve.
 \end{example}

\begin{example} \rm
Consider the genus 2 uniform dessin d'enfant $\mathcal{D} \subset X$ of type $(2,6,6)$ determined by the permutations
$$
\begin{array}{c}
\sigma_0  =  (1,7)(2,8)(3,9)(4,10)(5,11)(6,12)\\
\sigma_1  =  (1,2,3,4,5,6)(7,8,9,10,11,12)\\
\end{array}
$$
A straightforward computation gives 
$$
\sigma_1^3 \sigma_0=(1,10)(2,11)(3,12)(4,7)(5,8)(6,9).
$$ 
Thus, this dessin defines a multicurve $\Gamma=\{\gamma_1, \gamma_2, \gamma_3 \}$ (left hand side of Figure \ref{fig_y2=x6-1}) and, by Proposition \ref{pr:destomulticurve}, the minimum of the geodesic length function $\ell_{\Gamma}$ is achieved at the Grothendieck-Belyi suface $S_{\mathcal{D}}$ by a geodesic $\Gamma'=\varphi_{\mathcal{D}}(\Gamma) \subset S_{\mathcal{D}}$. It turns out that in this case  $S_{\mathcal{D}}$ corresponds to the complex algebraic curve $C:y^2=x^6-1$ and the Belyi function to the meromorphic function $\beta(x,y)=1-x^6$ (see Example 4.69 of \cite{Girondo_Gonzalez-Diez_2012}). This means that $S_{\mathcal{D}}$ consists of the affine points of the algebraic curve $C$ together with two \emph{points at infinity}, $\infty_1$ and $\infty_2$, which are the poles of $\beta$ and account for the centers of the two faces of the dessin $\mathcal{D}$ (or rather $\varphi_{\mathcal{D}}(\mathcal{D})$). Moreover, again by Proposition  \ref{pr:destomulticurve}, the multicurve $\Gamma'$ has $\beta^{-1}([0,1])$ as underlying graph. In order to visualize how the multicurve sits on the algebraic model of $S_{\mathcal{D}}$ we first note that the Belyi function $\beta:S_{\mathcal{D}} \longrightarrow \widehat{\mathbb{C}}$ only represents the quotient map  $\beta:S_{\mathcal{D}} \longrightarrow S_{\mathcal{D}}/G$, where $G$ is the automorphism group generated by $\tau(x,y)=(e^{2\pi i / 6}x,y)$ and $J(x,y)=(x,-y)$, the so-called \emph{hyperelliptic involution}. This is because $G$ is the  abelian group $C_6 \times C_2$ and $\beta$ is obviously a $G$-invariant function of  degree 12.

 \begin{figure}[!htbp]
\centering{
\begin{tabular}{ccc}
\includegraphics[width=0.45\textwidth]{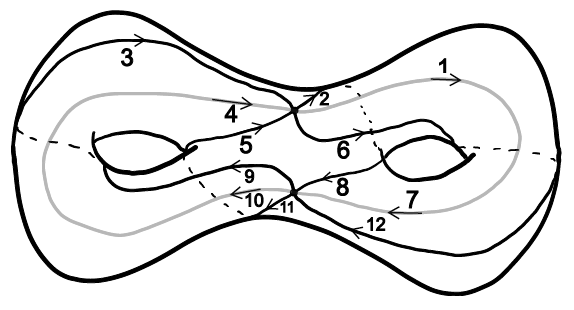} & \quad \quad \quad & 
\includegraphics[width=0.35\textwidth]{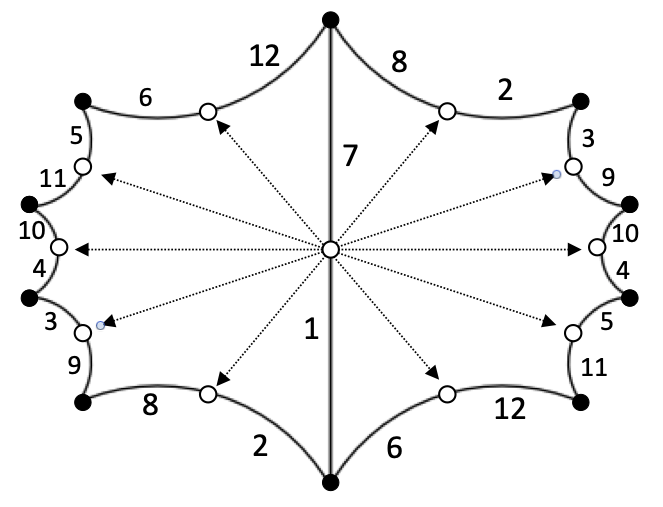} \ \\ 
\end{tabular}}
\caption{A length-minimizing filling curve in the compact Riemann surface of equation $y^2=x^6-1$.}\label{fig_y2=x6-1}
\end{figure}

Clearly, $\beta^{-1}(0)=\{w_k=(e^{2\pi i k/ 6},0), k=0, 1, \ldots , 5\}\subset C$ and $\beta^{-1}(1)=\{b^{\pm}:=(0, \pm i) \} \subset C$ are the white and black points of the dessin. The arc $I_1=\{ (\sqrt[6]{1-t}, i \sqrt{t} ) \ | \ 0 \le t \le 1\}\subset \beta^{-1}([0,1])$ is the edge of the dessin that connects the black vertex $b^+$ to the white vertex $w_0=(1,0)$ and $I_7=J(I_1)$ the edge that connects $w_0$ to the black vertex $b^-$. Now, as the cycle decomposition of $\sigma_1$ indicates, the black vertices have degree 6, so there must be 6 edges emanating from each of them, the ones corresponding to $b^-$ being $\tau^k(I_7)$, $k=0,1, \ldots , 5$. By the construction of the multicurve $\Gamma$ carried out in Proposition \ref{pr:destomulticurve}, after $I_7$ the next side we must traverse is the opposite to $I_7$, namely $I_4:=\tau^3(I_7)=(\tau^3 J)(I_1)$, which connects $b^-$ to $\tau^3(w_0)= w_3$, and the next one must be $I_{10}:=J(I_4)$, which leads us back to the vertex $J(b^-)=b^+$. So this sequence completes the first component
$$
\gamma_1=I_1\cup J(I_1) \cup (\tau^3 J)(I_1) \cup J(\tau^3 J)(I_1) 
$$
of our multicurve $\Gamma=\{\gamma_1, \gamma_2, \gamma_3\}$, the remaining ones being $\gamma_2=\tau(\gamma_1)$ and $\gamma_3=\tau^2(\gamma_1)$. We notice that the hyperbolic length of all these (Euclidean) unitary intervals $I_k \subset \mathbb{C}^2$ is 
$$
\cosh^{-1} \left( \displaystyle\frac{\cos \pi/6}{\sin \pi/12} \right)= \cosh^{-1} \left(\sqrt{3(2+\sqrt{3})} \right).
$$

\

As for the uniformization of our surface $S_{\mathcal{D}}\simeq \mathbb{D} / K$ we note that, since the numbering of edges can be chosen so as to match the action of $J$ and $\tau$ with the permutations $\sigma_0$ and $\sigma_1$, the group $K$, which is the kernel of the monodromy homomorphism $\omega: \Delta=\Delta(2,6,6) \longrightarrow \mathbb{S}_{12}$ determined by $x \longmapsto \sigma_0$, $y\longmapsto \sigma_1$, enjoys the property  that under the isomorphism $\Delta/ K \simeq G< \mathrm{Aut}(S_{\mathcal{D}})$ $x$ corresponds to $J$ and $y$ corresponds to $\tau$. Here we are using  the excepctional fact that because $|G|=\mathrm{deg}(\mathcal{D})=12$ the group $K= \ker \omega$ has already index 12 in $\Delta$, thus the groups $\ker (\omega) $ and $\omega^{-1}(\mathrm{Stab}_M(1))$ agree. (The right hand side of Figure \ref{fig_y2=x6-1} shows the Grothendieck-Belyi surface $S_{\mathcal{D}}$ as the quotient space of the action of certain side-pairing transformations on the fundamental domain of $K$).

\

Finally, the fact that $S_{\mathcal{D}}$ corresponds to the algebraic curve $C:y^2=x^6-1$  reflects our claim that the minimum of the geodesic length functions associated to uniform multicurves is attained at Riemann surfaces defined over a number field, which in this case is the field $\mathbb{Q}$ of rational numbers. This last property cannot be expected in general. In fact, since the \emph{absolute Galois group} $\mathrm{Gal}(\overline{\mathbb{Q}}|\mathbb{Q})$ is known to act faithfully on multicurves of given type $(2,2m,k)$  (\cite{Gonzalez-Diez_Jaikin_2015}), as the genus grows, some of the minimizing Riemann surfaces $S_{\mathcal{D}}$  will correspond to algebraic curves $F(x,y)= \sum a_{ij}x^i y^j=0$ whose coefficients $a_{ij}$ generate non-trivial field extensions of $\mathbb{Q}$.
\end{example}

\begin{example}[Dual curves] \rm
Let $\mathcal{D} \subset X$ a uniform dessin of type $(2,2m,2l)$ with corresponding Belyi pair $(S, \beta)$ so that $\mathcal{D}$ corresponds to a filling multicurve $\Gamma$ such that $\ell_{\Gamma}$ reaches its minimum at the hyperbolic surface $S$. Set $\widetilde{\beta}=\beta/(\beta -1)$. Then $\widetilde{\beta}$ is a new Belyi function on the same Riemann surface $S$ which interchanges the role of the black vertices (the fibre $\beta^{-1}(1)$) with the role of the face centers (the poles of $\beta$) while keeping the set of white vertices fixed. Corresponding to $\widetilde{\beta}$ there is a (\emph{dual}) uniform dessin $\widetilde{\mathcal{D}}$ and an associated filling multicurve $\widetilde{\Gamma}$ of type $(2,2l,2m)$. It is easy to see (see e.g. Section 2.2 in \cite{Gonzalez-Diez_Torres_2012}) that if $(\sigma_0, \sigma_1)$ is the permutation representation of $\mathcal{D}$, then the permutation representation of $\widetilde{\mathcal{D}}$ is $(\widetilde{\sigma_0}=\sigma_0, \widetilde{\sigma_1}=\sigma_{\infty}=(\sigma_0 \sigma_1)^{-1})$.

This observation can be used to produce different multicurves $\Gamma$ and $\widetilde{\Gamma}$ whose corresponding geodesic length functions $\ell_{\Gamma}$ and $\ell_{\widetilde{\Gamma}}$ reach their minimum at the same hyperbolic surface $S$. For instance if $\mathcal{D}$ (hence $\Gamma$) is determined by $\sigma_0=(1,2)(3,4)\cdots (11,12)$, $\sigma_1=(1,7,4,2,3,6)(5,11,8,9,12,10)$ then $\widetilde{\mathcal{D}}$ (hence $\widetilde{\Gamma}$) is determined by $\widetilde{\sigma_0}=\sigma_0$ and $\widetilde{\sigma_1}=(1,5,9,7,2,3)(4,8,12,10,11,6)$. 

Now 
$\sigma_1^3 \sigma_0=(3,6,9,8)(4,7,10,5)(1)(2)(11)(12)$ whereas $\widetilde{\sigma_1}^3 \widetilde{\sigma_0}=(1,5,12,8)(2,7,11,6)(3,10)(4,9)$; that is, $\Gamma$ has 3 components whereas its dual multicurve $\widetilde{\Gamma}$ has only 2.
\end{example}

\subsection{The non-clean uniform case} 

The graph underlying a filling multicurve $\Gamma$ can be sometimes bicoloured without adding extra vertices in the middle of the arcs of $\Gamma$. This happens when the graph underlying $\Gamma$ is bipartite. That is, the set of self-intersection points of $\Gamma$ splits as $\mathcal{B} \cup \mathcal{W}$, and every arc of $\Gamma$ connects a point in $\mathcal{B}$ and a point in $\mathcal{W}$. Of course, $\mathcal{B}$ and $\mathcal{W}$ are going to be the black and the white vertices of an obvious new dessin $\mathcal{D}^*_{\Gamma}$ that is well defined up to interchanging vertex colours. 

\

The following result is completely analogous to Proposition \ref{pr:destomulticurve}, and so it can be proved in a similar way:

\begin{proposition} \label{pr:destomulticurve_2}
Let  $\mathcal{D}\subset X$ be a  uniform dessin d'enfant of type $(2l,2m,j)$ and genus $g\ge 2$ with permutation representation $\sigma_0, \sigma_1$. Let $\beta: S=\mathbb{D} / K \longrightarrow \mathbb{D}/ \Delta$ be the corresponding Belyi pair and $\varphi_{\mathcal{D}}: X \longrightarrow S$ the associated homeomorphism. Then there exists a filling bipartite multicurve   $\Gamma'= \{\gamma'_1, \ldots, \gamma'_r \}$ on $S$ such that $\varphi_{\mathcal{D}}(\mathcal{D})=\beta^{-1}([0,1])\simeq \mathcal{D}^*_{\Gamma'}$. Moreover, the curves $\gamma'_i$ can be chosen to be geodesics of $S$ and the number $r$ of components of $\Gamma'$ agrees with half the number of cycles of $\sigma_1^m \sigma_0^l$.
\end{proposition}

We also have the following generalization of Theorem \ref{th:main}:

\begin{theorem} \label{th:main_2}
Let $X$ be a closed oriented topological surface of genus $g\ge2$, and let $\Gamma \subset X$ be a filling multicurve with $d$ arcs whose  underlying graph $\Gamma$ is bipartite. Suppose that the corresponding  dessin d'enfant  $\mathcal{D}^*_{\Gamma}$ is uniform of type $(2l, 2m, j)$. Then the geodesic length function $\ell_{\Gamma}$ reaches its minimum precisely at the Grothendieck-Belyi Riemann surface $S^*_{\Gamma}$ determined by the dessin  $\mathcal{D}^*_{\Gamma}$. Moreover if we denote by $\beta_*: S^*_{\Gamma} \to \widehat{\mathbb{C}}$ the corresponding Belyi function, the multigeodesic representing $\Gamma$ in $S^*_{\Gamma}$ has 
 $\beta_*^{-1}([0,1])$ as underlying graph and its length  is  
$$
\ell^*_{2l,2m,j,d} := d \cdot \cosh^{-1} \left(\displaystyle\frac{\cos\frac{\pi}{2m} \cos\frac{\pi}{2l} + \cos \frac{\pi}{j}}{\sin\frac{\pi}{2m} \sin\frac{\pi}{2l}} \right).
$$
\end{theorem}

The proof of this theorem follows the lines of the proof of Theorem \ref{th:main}. The expression for the minimal length $\ell^*_{2l,2m,j,d}$ comes from the fact that if $K^*<\Delta(2l,2m,j)$ is the Fuchsian representation of $\mathcal{D}^*_{\Gamma}$, the fundamental  domain of $K^*$   is a union of $n=d/j$ equilateral $2j$-gons with alternate angles $\pi/l$ and $\pi/m$.

\

An interesting remark can be made here. It may occur that the multicurve $\Gamma \subset X$ in Theorem \ref{th:main} whose associated dessin $\mathcal{D}_{\Gamma}$ is of type $(2,2m,k)$, has as underlying graph a bipartite graph thereby defining also a dessin $\mathcal{D}^*_{\Gamma}$, then necessarily of type $(2m,2m,k/2)$. Then we can also apply  Theorem \ref{th:main_2} to obtain the minimum of the function $\ell_{\Gamma}$. By the uniqueness of the minimum of $\ell_{\Gamma}$ we must have $\ell_{2m,k,d}=\ell^*_{2m,2m,k/2,d}$, and $S_{\Gamma}\simeq S^*_{\Gamma}$. The equality concerning length can be easily checked directly. But the identity of the hyperbolic surfaces $S_{\Gamma}=\mathbb{D}/K$ and $S^*_{\Gamma}=\mathbb{D}/K^*$ has a more interesting explanation. The point  is that, since the triangle group $\Delta(2m,2m,k/2)$ is contained in the triangle group $\Delta(2,2m,k)$, the inclusion $K^*<\Delta(2m,2m,k/2)$ giving the Fuchsian representation of $\mathcal{D}^*_{\Gamma}$ immediately gives $K^*<\Delta(2,2m,k)$ as the Fuchsian representation of $\mathcal{D}_{\Gamma}$. The combinatorial procedure that produces  $\mathcal{D}_{\Gamma}$ from $\mathcal{D}^*_{\Gamma}$ by declaring all vertices black and adding white vertices in the middle of the edges has been referred to as  \emph{medial surgery} in the literature (see \cite{Girondo_2003}).

\subsection{Homotopy invariance of the dessins defined by uniform filling multicurves} \label{sec:hominv}

While the Riemann surface minimizing the geodesic length function of a filling multicurve is, by definition, an invariant of the homotopy class of the multicurve, the dessin $\mathcal{D}_{\Gamma}$ (and, consequently, the Belyi-Grothendieck surface $S_{\Gamma}$) associated to a filling multicurve $\Gamma$ (in minimal position) is not; for instance, with a small deformation of the grey curve in Figure \ref{fig_y2=x6-1} one can obviously change the number of self-intersection points and the number of faces. Our last result in this section shows that the situation is more agreeable if we restrict ourselves to uniform dessins. To make our statement rigorous we introduce the term \emph{uniform filling multicurve} to refer to filling multicurves which admit a representative $\Gamma$ in minimal position such that one of the following conditions holds:
\begin{enumerate}
\item[1)]  $\mathcal{D}_{\Gamma}$ is a uniform dessin. 
\item[2)] The graph underlying $\Gamma$ is bipartite and $\mathcal{D}^*_{\Gamma}$ is a uniform dessin.
\end{enumerate}

\begin{corollary}
Let $\Gamma_1$ and $\Gamma_2$ be homotopically equivalent uniform filling multicurves of a closed oriented surface $X$ of genus $g\ge 2$, in minimal position. Then there exists $f \in \mathrm{Homeo}^+(X)$ such that $\Gamma_2=f(\Gamma_1)$, that is, $\mathcal{D}_{\Gamma_1} \simeq \mathcal{D}_{\Gamma_2}$ (resp. $\mathcal{D}^*_{\Gamma_1} \simeq \mathcal{D}^*_{\Gamma_2}$).
\end{corollary}

\begin{proof}
Let us assume that $\Gamma$ belongs to the case 1) above. The proof for the case 2) is exactly the same, replacing $S_i$ with $S^*_i$, $\mathcal{D}_{\Gamma_i}$ with  $\mathcal{D}^*_{\Gamma_i}$ and invoking Theorem \ref{th:main_2} instead of Theorem \ref{th:main}.

\

Let $(S_i, \beta_i)$ be the Belyi pair associated to the dessin $\mathcal{D}_{\Gamma_i}$, and $\varphi_i: X \longrightarrow S_i$ the associated homeomorphism that maps $\Gamma_i$ bijectively into $\beta_i^{-1}([0,1])$.

By Theorem \ref{th:main}, $\beta_i^{-1}([0,1])$ is the geodesic multicurve in $S_i$ which minimizes the geodesic length function $\ell_{\Gamma_i}: T_g \longrightarrow \mathbb{R}$. Since $\Gamma_1$ and $\Gamma_2$ are homotopically equivalent, the functions $\ell_{\Gamma_1}$ and $\ell_{\Gamma_2}$ agree. Hence, by uniqueness,  there is an isometry $\tau: S_1 \longrightarrow S_2$ such that $\beta_2^{-1}([0,1])= \tau \circ \beta_1^{-1}([0,1])$, that is $\varphi_2(\Gamma_2)= \tau \circ \varphi_1(\Gamma_1)$. It follows that $\varphi_2^{-1} \circ \tau \circ \varphi_1 \in \mathrm{Homeo}^+(X)$ solves our claim.
\end{proof}

This last result shows in particular that the type of a uniform filling multicurve is well defined.

\section{Uniform filling curves in general position}

Filling curves in general position are of special interest. Being in general position requires that the degree of each self-intersection point $p$ equals 4, that is $m_p=2$ with the notation we have been using throughout.
We devote this section to analyze some results about the existence of such curves. 
 
 \subsection{A surgery procedure}
 In his 2015 CUNY thesis \cite{Arettines_2015}, Arettines gave a surgery procedure to inductively construct in any genus $g$ a uniform filling curve $\gamma$ in general position whose complement has $n$ faces for $n=1$ (\cite{Arettines_2015}, Th. 4.3.1) and $n=2$  (\cite{Arettines_2015}, Th. 4.3.2). For $n=3$ this will only be possible when  $g\equiv 1 \pmod{3}$ as we note below. Similar constructions, for $n\ge 3$, have been given in \cite{Parsad-Sanki_2022} but all the examples there are non-uniform.
 
 Next we interpret this topological result in terms of the permutations $\sigma_0, \sigma_1$ that represent the monodromy of the corresponding dessins $\mathcal{D}_{\gamma}$. In view of what has gone before this will allow us to determine the Riemann surface at which the geodesic length functions of these filling curves attain their minima. Moreover, this approach permits us to solve the analogous question 
for the case of $n=3$ faces. 
  
 \
 
 Let $\mathcal{D}$ be a dessin of genus $g$, type $(2,4,k)$, degree $E$ and $n$ faces such that all white vertices have degree 2 and all black vertices have degree 4. Let 
 $(\sigma_0, \sigma_1)$ be the permutation representation of $\mathcal{D}$.
 
 Let $(a,b)$ be a 2-cycle of $\sigma_0$, such that $a$ and $b$ lie in different 4-cycles of $\sigma_1$. Thus, we can write the disjoint cycle decomposition of $\sigma_0$ and $\sigma_1$ as
$$
\begin{array}{c}
\sigma_0=(a,b)\sigma'_0\\
\ \\
\sigma_1=\sigma'_1(a,l_1, l_2, l_3)(b,l_4, l_5, l_6)
\end{array}
$$

Define now the two following elements of $\mathbb{S}_{E+8}$:
$$
\begin{array}{c}
\widetilde{\sigma_0}=\sigma'_0(a,E+5)(E+1,E+6)(E+2,E+7)(E+3,E+8)(E+4,b)\\
\ \\
\widetilde{\sigma_1}=\sigma_1(E+5,E+1,E+3,E+7)(E+4,E+2,E+6,E+8)
\end{array}
$$ 
which is the monodromy of a dessin $\widetilde{\mathcal{D}}$ with 4 additional white vertices, 2 additional black vertices and 8 additional edges compared to $\mathcal{D}$. The genus $\widetilde{g}$ of $\mathcal{D}$ is $g+1$ and, as we check below, the number of faces is still $n$, but either one face has increased the degree in 8 units or two faces have increased their degree in 2 and 6 units, respectively (see cases 1) and 2) below).

A direct computation shows that 
$$
\begin{array}{rcl}
\widetilde{\sigma_1} \widetilde{\sigma_0} & = & (a,l_1, l_2, l_3)(b,l_4, l_5, l_6)  (E+5,E+1,E+3,E+7)(E+4,E+2,E+6,E+8) \sigma'_1\\
& & \sigma'_0(a,E+5)(E+1,E+6)(E+2,E+7)(E+3,E+8)(E+4,b) \\
& & \\
& =  & (a,l_1, l_2, l_3)(b,l_4, l_5, l_6)(E+5,E+1,E+3,E+7)(E+4,E+2,E+6,E+8) \sigma'_1 \\
& &  \sigma'_0(a,b)(a,b)(a,E+5)(E+1,E+6)(E+2,E+7)(E+3,E+8)(E+4,b) \\
& & \\
& =  & (E+5,E+1,E+3,E+7)(E+4,E+2,E+6,E+8) \sigma_1\\
& &   \sigma_0(a,b)(a,E+5)(E+1,E+6)(E+2,E+7)(E+3,E+8)(E+4,b)
\end{array}
$$ 
and the latter is conjugated to \small
$$
\sigma_1   \sigma_0(a,b)(a,E+5)(E+1,E+6)(E+2,E+7)(E+3,E+8)(E+4,b)(E+5,E+1,E+3,E+7)(E+4,E+2,E+6,E+8)
$$ \normalsize
which equals
\begin{equation}\label{eq:prod}
\sigma_1 \sigma_0 (a, E+5, E+6, E+3, E+2, E+1, E+8)(b,E+4,E+7)
\end{equation}

We can distinguish now two possibilities:

1) If the edges $a$ and $b$ belong at both sides to the same face of $\mathcal{D}$ (e.g., the sides 4 and 13 in Figure \ref{fig_ej_planog2}), then 
$\sigma_1 \sigma_0$ decomposes as a disjoint product of cycles $\eta_1\cdots \eta_n$ where both $a$ and $b$ belong to, say, $\eta_n$. In this case (\ref{eq:prod}) equals
\small{$$
\eta_1 \cdots \eta_{n-1}(a,E+5,E+6,E+3,E+2,E+1,E+8,\eta_n(a), \eta_n^2(a), \ldots, b, E+4, E+7, \eta_n(b), \eta_n^2(b), \ldots, \eta_n^{-1}(a))
$$} \normalsize
 We see that in this case the dessin $\widetilde{\mathcal{D}}$  has degree $E+8$ and $n$ faces. There is an obvious one to one correspondence between faces of $\mathcal{D}$ and faces of $\widetilde{\mathcal{D}}$ that preserves all the degrees except one, that is increased in 8 units.

 2) If $a$ and $b$ belong to two different faces of $\mathcal{D}$ (e.g., the sides 1 and 16 in Figure \ref{fig_ej_planog2}), the cycle decomposition of $\sigma_1 \sigma_0$ has the form $\eta_1 \cdots \eta_{n-2}\eta_{n-1}, \eta_n$, with $a$ and $b$ belonging respectively to $\eta_{n-1}$ and $\eta_n$, say. In this case (\ref{eq:prod}) is 
 \small{$$
 \eta_1 \cdots \eta_{n-2}(a, E+5,E+6,E+3,E+2,E+1,E+8, \eta_{n-1}(a), \eta_{n-1}^2(a), \ldots , \eta_{n-1}^{-1}(a))(b, E+4, E+7, \eta_n(b), \eta_{n}^2(b),\eta_{n}^{-1}(b))
 $$} \normalsize
 In this case the obvious one to one correspondence between faces of $\mathcal{D}$ and faces of $\widetilde{\mathcal{D}}$  preserves  the degrees of $n-2$ of the $n$ faces. The degrees of the remaining two faces increase in 2 and 6 units respectively.
 
 Similar computations show that in both cases the number of cycles of $\widetilde{\sigma}_1^2 \widetilde{\sigma}_0$ agrees with the number of cycles of $\sigma_1^2 \sigma_0$.

\

We would like to point out that during the preparation of this work we learnt from Marston Conder that Darius Young has independently developed a group theoretical procedure similar to ours to construct a $g+1$ uniform clean dessin in general position out of a genus $g$ one.

 \subsection{Existence of uniform filling curves in general position with small number of faces}
 
If  $\gamma$ is a uniform filling (multi)curve in general position that decomposes the surface $X$ of genus $g\ge 2$ in $n$ faces with $k$ edges, a standard area computation yields
$$
(4g-4)\pi=n((k-2)\pi - k\frac{\pi}{2})
$$
that is 
\begin{equation} \label{eq:gnk}
8g-8=n(k-4)
\end{equation}

An inductive construction based on Arettines surgery may be used to show the existence of uniform filling multicurves up to 3 faces. Note that taking $n=1, 2, 3$ in (\ref{eq:gnk}) gives  $k=8g-4, 4g, (8g+4)/3$ respectively. Since $k$ must be an integer number, the genus $g\ge 2$ could be, in principle, arbitrary for $n=1$ or $n=2$, but restricts necessarily to $g\equiv 1\pmod{3}$ for $n=3$. The following result shows that there are indeed no further topological restrictions to the existence of this kind of uniform filling multicurves:

\begin{theorem} \label{th:exist}
1) For every genus $g\ge 2$ and for every $n\in\{ 1,2 \}$ there exists a uniform filling curve $\gamma$ in general position that decomposes the topological surface $X$ of genus $g$ into $n$ faces and such that the minimum of $\ell_{\gamma}$ is reached at a Riemann surface uniformized by a surface Fuchsian subgroup $K$ of a triangle group $\Delta(2,4,8g-4)$ (one face) or $\Delta(2,4,4g)$ (two faces).

2) For every genus $g>1$ congruent to 1 modulo 3, there exists a uniform filling curve $\gamma$ in general position that decomposes the topological surface $X$ of genus $g$ into three faces and such that the minimum of $\ell_{\gamma}$ is reached at a surface uniformized by a surface Fuchsian subgroup $K$ of a triangle group $\Delta(2,4,8g+4)/3)$.

In each case, the group $K$ is completely determined by the monodromy of the dessin associated to $\gamma$ and  the corresponding Riemann surface is defined over a number field.
\end{theorem}

\begin{proof}
1) We begin with the case $n=1$.

The surgery procedure described in the previous section, applied once to the uniform dessin  $\mathcal{D}_{\{\gamma\}}$ of type $(2,4,8g-4)$ and genus $g$ determined by a uniform filling curve $\gamma$ in general position readily gives a dessin $\widetilde{\mathcal{D}}$ which agrees with  $\mathcal{D}_{\{\widetilde{\gamma}\}}$, where $\widetilde{\gamma}$ is a filling curve in general position in the surface of genus $g+1$. Therefore, we only need to find an example in genus $g=2$ and argue by induction. 

Note that for genus $g=2$, $\gamma$ must be determined by a dessin d'enfant  of type $(2,4,12)$ and degree $12$ determined by a monodromy pair $(\sigma_0, \sigma_1)$  such that $\sigma_1^2\sigma_0$  decomposes into a product of two disjoint 6-cycles. In Example \ref{ex:(2,4,12)gen2} we have shown that there is one (and, in fact, only one) such dessin: the one with monodromy representation $(\sigma_0, \sigma_1)$ given by
$$\begin{array}{c}
 \sigma_0=(1,12)(2,11)(3,10)(4,9)(5,8)(6,7)\\ 
 \sigma_1=(1,4,7,10)(2,6,12,8)(3,5,11,9)
 \end{array}
 $$

\

For the case $n=2$ we need to be careful with the fact that the surgery, applied once to a uniform dessin of genus $g$ in case 2) of the previous section, produces a dessin of genus $g+1$ that is not uniform, since one of the faces increases the degree by 2 and the other one by 6. But a second surgery applied to it (with the obvious reversal choice of which face degree will be increased by 2 and which one will be increased by 6) produces a uniform dessin of genus $g+2$. Finding examples for $g=2$ and $g=3$ is enough to finish the proof by induction, and the filling curves in Figure \ref{fig_g2a} do the job. The monodromy representation pair $(\sigma_0, \sigma_1)$ is 

$$
\begin{array}{c}
\sigma_0 =(1,16)(2,15)(3,14)(4,13)(5,12)(6,11)(7,10)(8,9)\\
  \sigma_1=(1,6,9,12)(2,10,16,8)(3,13,15,5)(4,7,14,11)
 \end{array}
 $$
 
in the case $g=2$, and 
$$\begin{array}{c}
\sigma_0=(1,13)(2,24)(3,23)(4,22)(5,21)(6,20)(7,19)(8,18)(9,17)(10,16)(11,15)(12,14)\\ 
\sigma_1=( 1, 3, 14, 24)( 2, 4, 13, 23)( 5,9,22,18)(6,16,21,11)(7,17,20,10)(8,15,19,12)
\end{array}
$$ 
in the case $g=3$.

2) Note that three suitable consecutive applications of the surgery procedure increase the genus in 3 units, and may be performed in such a way that the final dessin is uniform if the starting one was. Thus, we only need one initial example (now for $g=4$) to finish the proof inductively. The monodromy of the required $g=4$ dessin can be taken as 
$$\begin{array}{rcl}
\sigma_0&=&( 1, 2)( 3, 4)\cdots (35,36)\\
\sigma_1&=&(1,3,9,5)(2,4,6,7)(8,11,17,13)(10,12,19,15)(14,21,29,23)(16,25,24,27)(18,26,
31,20)\\
& & (22,33,35,28)(30,32,34,36) \mbox{\hspace{8.3cm}}
\end{array}
$$
 \end{proof}

The situation is different for $n=4$. For instance, the case $g=2$, $n=4$ corresponds to $k=6$. Using computer algebra software such as \cite{GAP4}, one can show that there exist forty uniform dessins of type $(2,4,6)$, genus $g=2$ and degree 24. By Proposition \ref{pr:destomulticurve} all of them define multicurves that fill up the surface of genus 2 decomposing it into $n=4$ topological hexagons. In other words, there are forty non-conjugate monodromy pairs $(\sigma_0, \sigma_1)$ of permutations in $\mathbb{S}_{24}$ such that $\sigma_0$ is a product of twelve disjoint transpositions, $\sigma_1$ is a product of six disjoint $4$-cycles and $\sigma_1 \sigma_0$ is a product of four disjoint $6$-cycles. 

Nevertheless, it can be directly checked that $\sigma_1^2 \sigma_0$ never splits as a product of just two disjoint $12$-cycles, hence none of these dessins corresponds to a filling curve. 

We have thus proved the following

\begin{proposition} \label{pr:all_multi}
Every uniform multicurve that fills up the surface of genus $g=2$ and decomposes it into $n=4$ topological hexagons has more than one component.
\end{proposition}

That is, sometimes the family   
$$
\begin{array}{c}
 \left\{ \Gamma \ | \ \Gamma \mbox{ is a uniform filling multicurve on the surface of }   \right. \\
  \left. \qquad \  \mbox{genus } g  \mbox{ and decomposes it into } n \mbox { faces of } k \mbox{ edges}  \right\}
\end{array}
$$
contains only proper multicurves. In other words, there are   topological restrictions to the existence of filling curves in general position.
%
%
%
%

%
%

\section{An upper-bound of $\ell_{\Gamma}$ in the non-uniform case}
Assume that the  multicurve $\Gamma$ is in minimal position and fills up the closed oriented surface $X$ of genus $g\ge 2$ but is not a uniform filling multicurve.

As before, set $S_{\Gamma}=\mathbb{D}/K$ and let $\widetilde{K}$ be a torsion free finite index subgroup of $K$, which exists by Selberg's Lemma. There is a commutative diagram as follows:

%
%
%
$$
\xymatrix{
\widetilde{S_{\Gamma}} = \mathbb{D}/\widetilde{K}  \ar[r]^{p}    \ar[dr]_{\widetilde{\beta}} & 
S_{\Gamma}=\mathbb{D}/K \ar[d]^{\beta}  \\
& \widehat{\mathbb{C}} =\mathbb{D}/\Delta(2,2m,k) 
}
$$
where now $\widetilde{K}$ is torsion free but $K$ is not. Let $\ell_{m,k}$ the length of the side 
of the triangle with angles $\pi/2, \pi/2m $ and $\pi/k$ opposite to this last angle.

By Theorem \ref{th:main}, $\widetilde{\Gamma}=p^{-1}(\Gamma')$, with $\Gamma'=\beta^{-1}([0,1])=\varphi_{\Gamma}(\Gamma)$, minimizes the corresponding geodesic length function, and its length equals
$$
\mathrm{length}(\widetilde{\Gamma})= \mathrm{deg}(\widetilde{\beta})  \ell_{m,k}= \mathrm{deg}(p) \mathrm{deg}(\beta) \ell_{m,k}.
$$ 

On the other hand, by Schwarz-Pick Theorem (see \cite{McMullen} for the precise version of this classical result we need) we have
$$
\mathrm{length}(\widetilde{\Gamma}) > \mathrm{deg}(p) \mathrm{length}(\Gamma')
$$
hence
$$
 \mathrm{deg}(\beta) \ell_{m,k}>  \mathrm{length}(\Gamma').
$$

In other words, among the multicurves $\Gamma$ with given number of arcs and given  type (of the associated dessin $\mathcal{D}_{\Gamma}$) the uniform ones achieve the maximum of the minima of the functions  $\ell_{\Gamma}$.

This proves the following result:

\begin{theorem} \label{th:nounif}
Let $\Gamma$ be a multicurve with $d$ arcs that fills up a closed oriented surface $X$ of genus $g\ge 2$ and is not a uniform filling curve. Assume that $\mathcal{D}_{\Gamma}$ has type $(2,2m,k)$. Then the minimum of the geodesic length function $\ell_{\Gamma}$ is strictly smaller than  
$2d \cosh^{-1} \left( \displaystyle\frac{\cos \pi/k}{\sin \pi/2m} \right)$.
\end{theorem}

\

\noindent {\bf Acknowledgements}.- We are indebted to M. Conder, J. Parker and J. Souto for valuable discussions during the preparation of this paper.

E. Girondo and G. Gonz\'alez-Diez supported by grants CEX2019-000904-S and PID2019-106617GB-I00, funded by MCIN/AEI/10.13039/501100011033, and by the Madrid Government (Comunidad de Madrid – Spain) under the multiannual Agreement with UAM in the line for the Excellence of the University Research Staff in the context of the V PRICIT (Regional Programme of Research and Technological Innovation).

%
%
%
%
Rub\'en Hidalgo supported by Agencia Nacional de Investigaci\'on y Desarrollo (ANID) research project FONDECYT 1230001.

\bibliography{references}  
\bibliographystyle{alpha}

\end{document}